\documentclass{amsart}

\usepackage{a4}
\usepackage{amsmath}
\usepackage{amsthm}
\usepackage{calc}
\usepackage{color}
\usepackage{delarray}
\usepackage{enumerate}
\usepackage{epsfig}
\usepackage{exscale}
\usepackage{graphicx}
\usepackage{latexsym}
\usepackage{multicol}
\usepackage{nextpage}
\usepackage{pifont}
\usepackage{rawfonts}
\usepackage{rotating}
\usepackage{url}
\usepackage{verbatim}
\usepackage{xspace}

\theoremstyle{plain}
  \newtheorem{theorem}{Theorem}[section]
  \newtheorem*{theorem*}{Theorem}
  \newtheorem{corollary}[theorem]{Corollary}
  \newtheorem*{corollary*}{Corollary}
  \newtheorem{lemma}[theorem]{Lemma}
  \newtheorem*{lemma*}{Lemma}
  
  \newtheorem*{proposition*}{Proposition}

\theoremstyle{definition}
  \newtheorem{algorithm}[theorem]{Algorithm}
  \newtheorem*{algorithm*}{Algorithm}
  
  \newtheorem*{assumption*}{Assumption}
  
  \newtheorem*{conjecture*}{Conjecture}
  
  \newtheorem*{definition*}{Definition}
  
  \newtheorem*{example*}{Example}
  
  \newtheorem*{hypothesis*}{Hypothesis}
  
  \newtheorem*{property*}{Property}
  \newtheorem{remark}[theorem]{Remark}
  \newtheorem*{remark*}{Remark}

\newcommand{\figperso}[2]{\begin{figure}[ht]\begin{center} #1
\\[-2em] \
\end{center}
#2\end{figure}}

\newcommand{\R}{\mathbb R}

\newcommand{\Sy}{\mathbb S}
\newcommand{\E}{\mathbb E}
\newcommand{\M}{\mathcal M}
\newcommand{\T}{\mathcal T}
\newcommand{\U}{\mathcal U}
\newcommand{\F}{\mathcal F}
\newcommand{\Q}{\mathcal Q}
\newcommand{\G}{\mathcal G}
\newcommand{\W}{\mathcal W}

\newcommand{\D}{\mathcal D}
\newcommand{\Po}{\mathcal P}
\newcommand{\Li}{\mathcal L}
\newcommand{\Ha}{\mathcal H}
\newcommand{\C}{\mathcal C}
\newcommand{\cor}{\mathop{\mathsf{Cor}}}
\newcommand{\tr}{\mathop{\mathrm{tr}}}

\newcommand{\tp}{^{\mathsf T}\,}
\newcommand{\card}[1]{\# #1}

\newcommand{\A}{\mathfrak A}
\newcommand{\NXW}{N_{\mathrm{rg}}}
\newcommand{\NX}{N_x}
\newcommand{\NW}{N_w}
\newcommand{\Nzero}{N_{\mathrm{in}}}

\newcommand{\eps}{\epsilon}
\newcommand{\NEW}[1]{{\em #1\/}\index{#1}}
\newcommand{\Msmall}{\overline{\M}}
\newcommand{\cardMs}{\card{\Msmall}}
\newcommand{\cardM}{\card{\M}}
\def\<#1,#2>{\langle #1, #2\rangle}

\title[A probabilistic max-plus scheme for solving HJB equations]{From a monotone probabilistic scheme to a probabilistic max-plus
algorithm for solving Hamilton-Jacobi-Bellman equations}

\author{Marianne Akian \and Eric Fodjo}
\address{M.~Akian: INRIA and CMAP, \'Ecole Polytechnique CNRS. Address:
CMAP, \'Ecole Polytechnique,
Route de Saclay,
91128 Palaiseau Cedex}
\email[M.~Akian]{Marianne.Akian@inria.fr}
\address{E.~Fodjo: I-Fihn Consulting and 
INRIA and CMAP, \'Ecole polytechnique CNRS. Address:
CMAP, \'Ecole Polytechnique,
Route de Saclay,
91128 Palaiseau Cedex}
\email[E.~Fodjo]{eric.fodjo@polytechnique.edu}



\thanks{The first author was partially supported by the ANR project MALTHY, ANR-13-INSE-0003,
by ICODE, and by PGMO, a joint program of EDF and FMJH (Fondation Math{\'e}matique Jacques Hadamard)}

\begin{document}

\begin{abstract}
In a previous work (Akian, Fodjo, 2016),
we introduced a lower complexity 
probabilistic max-plus numerical method for solving 
fully nonlinear Hamilton-Jacobi-Bellman equations
associated to diffusion control problems involving 
a finite set-valued (or switching) control and possibly 
a continuum-valued control.
This method was based on the idempotent
expansion properties obtained by McEneaney, Kaise and Han (2011) 
and on the numerical probabilistic method 
proposed by Fahim, Touzi and Warin (2011)
for solving some fully nonlinear parabolic partial differential
equations.
A difficulty of the algorithm of Fahim, Touzi and Warin is in the critical
constraints imposed on the Hamiltonian to ensure the monotonicity of the
scheme, hence the convergence of the algorithm.
Here, we propose a new ``probabilistic scheme'' which is monotone
under rather weak assumptions, including the case of 
strongly elliptic PDE with bounded coefficients.
This allows us to apply our probabilistic max-plus method
in more general situations.
We illustrate this on the 
evaluation  of the superhedging price of an option under uncertain 
correlation model
with several underlying stocks and changing sign cross gamma,
and consider in particular
the case of 5 stocks leading to a PDE in dimension 5.
\end{abstract}

\maketitle

\section{Introduction}

\label{sec-int}

We consider a finite horizon diffusion control problem on $\R^d$ involving 
at the same time a ``discrete'' control taking its values in a finite set $\M$, 
and a ``continuum'' control taking its values in some subset $\U$ 
of a finite dimensional space $\R^p$ (for instance a 
convex set with nonempty interior), which we next describe.

Let $T$ be the horizon. 
The state $\xi_s\in\R^d$ at time $s\in [0,T]$
satisfies the stochastic differential equation
\begin{equation}\label{defxi}
d \xi_s = f^{\mu_s} (\xi_s, u_s) ds + \sigma^{\mu_s} (\xi_s, u_s) d W_s
\enspace , \end{equation}
where $(W_s)_{s\geq 0}$ is a $d$-dimensional
Brownian motion on a filtered probability space
$(\Omega,\F,(\F_s)_{0\leq s\leq T},P)$.
The control processes $\mu:=(\mu_s)_{0\leq s\leq T}$ and $u:=(u_s)_{0\leq s\leq T}$
take their values in the
sets $\M$ and $\U$ respectively and they are admissible if
they are  progressively measurable with respect to
the filtration $(\F_s)_{0\leq s\leq T}$.
We assume that, for all $m\in\M$,
the maps $f^m: \R^d\times \U\to \R^d$ and
$\sigma^m: \R^d\times \U\to \R^{d\times d}$ are continuous
and satisfy properties implying the existence of the process 
$(\xi_s)_{0\leq s\leq T}$ for any admissible control processes $\mu$ and $u$.

Given an initial time $t\in [0,T]$, the control problem consists in maximizing
the following payoff:
\begin{align*}
J(t, x, \mu, u) :=& \E \left[ \int_t^T 
e^{ - \int_t^s \delta^{\mu_{\tau}} (\xi_{\tau}, u_{\tau}) d{\tau} }
\ell^{\mu_s} (\xi_s, u_s) ds \right.\\
& \left. \qquad + e^{ - \int_t^T \delta^{\mu_{\tau}} (\xi_{\tau}, u_{\tau}) d{\tau} }
\psi(\xi_T)  \mid \xi_t = x \right] \enspace ,
\end{align*}
where, for all $m\in \M$, $\ell^m: \R^d\times \U\to \R$,
$\delta^m: \R^d\times \U\to \R$,
and $\psi:\R^d\to\R$ are  given continuous maps.
We then define the value function of the problem as the optimal payoff:
$$v(t, x) = \sup_{\mu, u} J(t,x,\mu, u)\enspace ,$$
where the maximization holds over all admissible control processes
$\mu$ and $u$.

Let $\Sy_d$ denotes the set of symmetric $d\times d$ matrices
and let us denote by $\leq$ the Loewner order on $\Sy_d$
($A\leq B$ if $B-A$ is nonnegative).
The Hamiltonian $\Ha:\R^d\times \R \times \R^d\times \Sy_d\to \R$
of the above control problem is defined as:
\begin{subequations}\label{defH}
\begin{align}
\Ha(x, r,p,\Gamma):=&\max_{m\in \M} \Ha^m(x, r,p,\Gamma)\enspace,
\label{defHmax}\end{align}
with
\begin{align}
 \Ha^m(x, r,p,\Gamma):=&\max_{u\in \U}  \Ha^{m,u}(x, r,p,\Gamma) 
\enspace ,\label{defHm}\\
 \Ha^{m,u}(x, r,p,\Gamma):=&
\frac{1}{2} \tr\left(\sigma^m(x,u) \sigma^m(x,u)\tp \Gamma
\right)  +f^m(x,u)\cdot p\nonumber\\
&\qquad -\delta^m(x,u)r +\ell^m(x,u) \enspace.
\label{defHmu}
\end{align}
\end{subequations}

Under suitable assumptions, the value function 
$v:[0,T]\times \R^d\to \R$ is the unique 
(continuous) viscosity solution of
the following Hamilton-Jacobi-Bellman equation
\begin{subequations}\label{HJB}
\begin{align}
&-\frac{\partial v}{\partial t} 
-\Ha(x, v(t,x), Dv(t,x), D^2v(t,x))=0, 
\quad  x\in \R^d,\; t\in [0,T), \label{HJB1}\\
& v(T,x)=\psi(x), \quad x\in \R^d,
\end{align}
\end{subequations}
satisfying also some growth condition at infinity (in space).

In~\cite{touzi2011}, Fahim, Touzi and Warin proposed a probabilistic numerical 
method to solve such fully nonlinear partial 
differential equations~\eqref{HJB}, inspired by their backward stochastic
differential equation interpretation 
given by Cheridito, Soner, Touzi and Victoir
in~\cite{cheridito2007}.
In~\cite{touzi2011}, the convergence of the resulting algorithm
follows from the theorem of Barles and Souganidis~\cite{barles90},
which requires the monotonicity of the scheme.
Moreover, for this monotonicity to hold, critical
constraints are imposed on the Hamiltonian: 
the diffusion matrices $\sigma^m(x,u) \sigma^m(x,u)\tp$ need
at the same time to be bounded from below (with respect to the Loewner order)
by a symmetric positive definite matrix $a$ and 
bounded from above by $(1+2/d) a$.
Such a constraint can be restrictive, in particular it may not hold
even when the matrices  $\sigma^m(x,u)$ do not depend on $x$ and $u$ but take
different values for $m\in \M$. %
In~\cite{monotone-zhang2015}, Guo, Zhang and Zhuo proposed a monotone
scheme exploiting the diagonal part of the diffusion matrices 
and combining a usual finite
difference scheme to the scheme of~\cite{touzi2011}.
This new scheme can be applied in more general situations than the one 
of~\cite{touzi2011}, but still does not work for general 
control problems.

McEneaney, Kaise and Han
proposed in~\cite{mceneaney2010,mceneaney2011}
an idempotent numerical method which works at least when 
the Hamiltonians with fixed discrete control, $\Ha^m$,
correspond to linear quadratic control problems.
This method is based on the distributivity of the (usual) addition operation
over the supremum (or infimum) operation, and on a property
of invariance of the set of quadratic forms.
It computes in a backward manner the value function $v(t,\cdot)$
at time $t$ as a supremum of quadratic forms.
However, as $t$ decreases, the number of quadratic forms generated 
by the method increases exponentially (and even become infinite
if the Brownian is not discretized in space) and some pruning 
is necessary to reduce the complexity of the algorithm.

In~\cite{fodjo1}, we introduced an algorithm
combining  the  two above methods which uses in particular 
the simulation of as many uncontrolled stochastic processes as 
discrete controls. Moreover, we shown that
even without pruning, the complexity of the algorithm
is bounded polynomially in the number of discretization time steps
and the sampling size.

However, due to the above critical constraints imposed in~\cite{touzi2011},
the algorithm of~\cite{fodjo1} is difficult to apply in practical situations.
One way to avoid these critical constraints, is
as suggested in~\cite{fodjo1}, to introduce a large number of
Hamiltonians $\Ha^m$ such that each of them satisfy the constraints.
Since one need to simulate a stochastic process for each Hamiltonian,
this technique may increase the complexity unnecessarily.

Here, we propose a different probabilistic discretization of
the Hessian of the value function, which ensures the monotonicity 
of the scheme in rather general situations,
including the case of strongly elliptic PDE with bounded coefficients
and we show how the algorithm of~\cite{fodjo1} associated to the new scheme
can be applied in these situations and high dimension.

The paper is organized as follows.
In Section~\ref{sec-touzi}, we recall the scheme of~\cite{touzi2011}.
Then, the new monotone probabilistic discretization is presented in
Section~\ref{sec-monotone}.
In Section~\ref{sec-maxplus}, we recall the algorithm of~\cite{fodjo1}
and show how it can be combined with the scheme of
Section~\ref{sec-monotone}.
In Section~\ref{sec-illust}, we illustrate this algorithm numerically.
There, we consider the 
evaluation of the superhedging price of an option under uncertain 
correlation model
with several underlying stocks and changing sign cross gamma.
We consider in particular the case of 5 stocks
leading to a PDE in dimension 5.

\section{The probabilistic method of Fahim, Touzi and Warin}
\label{sec-touzi}

In the present section we recall the probabilistic numerical 
method of Fahim, Touzi and Warin proposed in~\cite{touzi2011}.
We begin with the general description and continue with an example
in order to illustrate the critical constraint.

\subsection{General description}

Let $h$ be a time discretization step such that $T/h$ is an integer.
We denote by $\T_h=\{0,h,2h,\ldots, T-h\}$ the set of discretization times
of $[0,T)$.
Let $\Ha$ be any Hamiltonian of the form~\eqref{defH}.
Let us decompose $\Ha$ as the sum $\Ha=\Li+\G$ of the (linear) generator 
$\Li$ of a given diffusion (with no control) and of
a nonlinear elliptic Hamiltonian $\G$. 
This means that 
\begin{align*}
 \Li(x, p,\Gamma):=&
\frac{1}{2} \tr\left(a(x)  \Gamma \right) +\underline{f}(x)\cdot p \enspace,
\end{align*}
with $a(x)=\underline{\sigma}(x) \underline{\sigma}(x)\tp$,
for some drift map $x\in\R^d\mapsto \underline{f}(x)\in \R^d$  and
standard deviation (volatility) map 
$x\in\R^d\mapsto \underline{\sigma}(x)\in\R^{d\times d}$.
This also means that the Hamiltonian $\G=\Ha-\Li$ satisfies the
ellipticity condition, that is
$\partial_{\Gamma} {\G}(x, r, p,\Gamma)$ 
is positive semidefinite, for all $x\in\R^d,\;
r\in\R, p\in \R^d, \Gamma\in \Sy_d$.
We shall also assume that $a(x)$ is positive definite (which implies 
that $\underline{\sigma}(x)$ is invertible). 
A typical example is obtained when $\Ha$ is uniformly 
strongly elliptic and $\Li(x, p,\Gamma)=
\frac{\epsilon}{2} \tr\left(  \Gamma \right) $
with $\epsilon$ small enough,
which corresponds to the generator $\frac{\epsilon}{2}\Delta v$.
Denote by $\hat{X}$
the Euler discretization of the diffusion with generator $\Li$:
\begin{equation}
\hat{X}(t+h)=
 \hat{X}(t)+ \underline{f}(\hat{X}(t))h+
\underline{\sigma}(\hat{X}(t)) (W_{t+h} -W_t)\enspace .
\label{xhat} \end{equation}
The time discretization of \eqref{HJB} proposed in~\cite{touzi2011} 
has the following form:
\begin{equation}\label{scheme}
 v^h(t,x)=T_{t,h}(v^h(t+h,\cdot))(x),\quad t\in\T_h\enspace,\end{equation}
with 
\begin{equation}
T_{t,h}(\phi)(x)= 
 \D_{t,h}^{0}(\phi)(x)
+ h \G(x, \D_{t,h}^{0}(\phi)(x),\D_{t,h}^{1}(\phi)(x),\D_{t,h}^{2}(\phi)(x))
\enspace .
\label{def-th-m}
\end{equation} 
In~\eqref{def-th-m}, $ \D_{t,h}^i(\phi)$, $i=0,1,2$, denotes the following 
approximation of the $i$th differential of $e^{h\Li}\phi$:
\begin{subequations}\label{defDcal}
\begin{align}
 \D_{t,h}^i(\phi)(x)&=\E(D^i\phi(\hat{X}(t+h))\mid \hat{X}(t)=x)\enspace ,
\end{align}
where $D^i$ denotes the $i$th differential operator.
Moreover, it is computed using the following scheme:
\begin{align}
 \D_{t,h}^i(\phi)(x)&=\E(\phi(\hat{X}(t+h))
\Po^i_{t,x,h}(W_{t+h}-W_t) \mid \hat{X}(t)=x)\label{defDcal2}
\enspace ,
\end{align}
\end{subequations}
where, for all $t,x,h,i$, $\Po^i_{t,x,h}$ is
the polynomial of degree $i$  in the variable $w\in \R^d$
given by:
\begin{subequations}\label{defpol}
\begin{align}
 \Po^0_{t,x,h}(w)&=1\enspace,\\ \Po^1_{t,x,h}(w)&=(\underline{\sigma}(x)\tp)^{-1}h^{-1} w\enspace,\\
 \Po^2_{t,x,h}(w)&=(\underline{\sigma}(x)\tp)^{-1} h^{-2}(w w\tp-h I)
(\underline{\sigma}(x))^{-1}  \label{defpol3}
\enspace ,
\end{align}
\end{subequations}
where $I$ is the $d\times d$ identity matrix.
Note that the equality between the two formulations in~\eqref{defDcal} holds 
for all $\phi$ with exponential
growth~\cite[Lemma 2.1]{touzi2011}.

In addition to the formal expression in~{\rm (\ref{def-th-m}-\ref{defpol})},
which can be compared to a standard numerical approximation (or more 
precisely to a time discretization),
the method of~\cite{touzi2011} includes the approximation of the 
conditional expectations in~\eqref{defDcal} by any probabilistic method such as
a regression estimator:
after a simulation of the processes $W_t$ and $\hat{X}(t)$,
one apply at each time $t\in\T_h$ a regression estimation to find the value
of $\D_{t,h}^i(v^h(t+h,\cdot))$ at the points $\hat{X}(t)$ by using
the values of  $v^h(t+h,\hat{X}(t+h))$ and $W_{t+h}-W_t$.
Hence, although in our setting the operator $T_{t,h}$ 
does not depend on $t$, since both the law of $W_{t+h}-W_t$ 
and the Hamiltonian $\Ha$ do not depend on $t$,
we shall keep the index $t$ in the above expressions to allow 
further approximations as above.

In~\cite{touzi2011}, the convergence of the time discretization 
scheme~\eqref{scheme} is 
proved by using the theorem of Barles and Souganidis of~\cite{barles90},
under the above assumptions together with the critical
assumption that
$\partial_{\Gamma} {\G}(x, r, p,\Gamma)$ is lower bounded by
some positive definite matrix (for all $x\in\R^d,\;
r\in\R, p\in \R^d, \Gamma\in \Sy_d$)
and that \sloppy
$\tr(a(x)^{-1} \partial_{\Gamma} {\G})\leq 1$.
Indeed, the latter conditions together with the boundedness of 
$\partial_{p} {\G}$ are used to show (in Lemma 3.12 and 3.14) that the operator
 $T_{t,h}$ is a $Ch$-almost monotone operator
over the set of Lipschitz continuous functions from $\R^d$ to $\R$,
where we shall say that an operator $T$ between any 
partially ordered sets $\F$ and $\F'$
of real valued functions (for instance $\R^n$, or the set of bounded functions 
from some set $\Omega$ to $\R$)
is \NEW{$L$-almost monotone}, for some constant $L\geq 0$, if 
\begin{equation}\label{amonotone}
\phi,\psi\in \F,\; 
 \phi\leq \psi \implies T(\phi)\leq T(\psi)+
L \sup(\psi-\phi)\enspace ,\end{equation}
and that it is \NEW{monotone}, when this holds for $L=0$.

Note that some other technical assumptions are used in~\cite{touzi2011},
such as the uniform Lipschitz continuity of the Hamiltonian $\Ha$, 
and the boundedness of 
the value function of the corresponding control problem,
which are less crucial, since they can be replaced by more 
usual stochastic control assumptions 
(like boundedness and Lipschitz continuity of the coefficients of
the controlled diffusion itself).

In~\cite{fodjo1}, we proposed to bypass the critical constraint, by
assuming that the Hamiltonians $\Ha^m$ (but not necessarily $\Ha$)
satisfy the critical constraint, 
and applying the above scheme to the Hamiltonians $\Ha^m$.
Another way is to replace the $\G$ part of the operator~\eqref{def-th-m} 
by any approximation of it in $O(h)$, for instance by using or combining
the probabilistic scheme with a finite difference scheme, as  is done
in~\cite{monotone-zhang2015}.
Indeed, the operator~\eqref{def-th-m} is already an
approximation of the semigroup of the HJB equation in time $h$ which
is at best in $O(h^2)$, therefore one can replace, with no loss of order of 
approximation, the $\D_{t,h}^i(\phi)$ inside $\G$ 
by $D^i \phi(x)$ (which is an approximation in $O(h)$)
or any approximation of order $O(h)$ of it.
Note however that the first $\D_{t,h}^0(\phi)$ in~\eqref{def-th-m}
 can only be replaced  by an approximation in $O(h^2)$ or at least in $o(h)$.
In Section~\ref{sec-monotone}, we shall propose an
approximation of  $\D_{t,h}^2(\phi)$ or $D^2 \phi(x)$
which is expressed as a conditional expectation as in~\eqref{defDcal2},
and leads to a monotone operator $T_{t,h}$ 
without assuming the critical constraint.
The advantage with respect to finite difference methods or with the method
of~\cite{monotone-zhang2015} is that it can still be used with simulations.
Before describing the new scheme, we shall compare on some examples
the method of~\cite{touzi2011} with finite difference schemes.

\subsection{Examples and comparison with finite difference schemes
constraints}
\label{remark-k}
Let us first show on examples the behavior of the discretization 
of~\cite{touzi2011}.
This should help to understand the advantage of the new discretization
that we propose in next section.
For this, we shall show what happen when the increments of the Brownian 
motion $W_{t+h}-W_t$
are replaced by any finite valued independent random variables with same law.
This allows one in particular to compare the discretization 
of~\cite{touzi2011} with finite difference schemes.
Similar comparisons were done  in~\cite{touzi2011} but here we shall 
discuss in addition the meaning of the critical constraint in this situation.

To simplify the comparison, consider the case where $\Ha$ is linear
and depends only on $\Gamma$, that is
\[ \Ha(x,r,p,\Gamma)= \frac{1}{2} \tr\left(A \Gamma \right)\]
where $A$ is a $d$-dimensional symmetric positive definite matrix.
We assume that $A\geq I$ and choose 
$\Li(x, p,\Gamma)=\frac{1}{2} \tr\left(\Gamma \right)$, that is
$\underline{f}\equiv 0$ and $\underline{\sigma}\equiv I$.
Hence, $\G(x,r,p,\Gamma)= \frac{1}{2} \tr\left((A-I) \Gamma \right)$.

Then, denoting by $N$ any $d$-dimensional normal random variable,
we get that the operator $T_{t,h}$ of~\eqref{def-th-m} satisfies:
\begin{align}
T_{t,h}(\phi)(x)&=  \D_{t,h}^{0}(\phi)(x)
+ h \frac{1}{2} \tr ((A-I)\D_{t,h}^{2}(\phi)(x)) )\nonumber\\
&=\E\left(\phi(x+\sqrt{h} N)
(1+\frac{1}{2} \tr ((A-I) (NN\tp -I)))\right)%
\enspace .
\label{ex-th}
\end{align}
This operator is linear, and it is thus monotone if and only if
for almost all values of $N$ the coefficient of $ \phi(x+\sqrt{h} N)$
inside the expectation, 
that is $(1+\frac{1}{2}\tr((A-I) (NN\tp-I)))$,
is nonnegative.
The critical constraint  $\tr(a(x)^{-1} \partial_{\Gamma} {\G})\leq 1$
is equivalent here to $\frac{1}{2}\tr(A-I)\leq 1$.
This corresponds exactly to the condition that the coefficient of  $ \phi(x)$ 
inside the expectation is nonnegative.
Thus, if $N$ is replaced by any random variable taking a finite number of values
including $0$, the critical constraint is necessary.

Consider the dimension $d=1$ and a simple discretization of $N$
by the random variable taking the values $\pm \nu$ with probability
 $1/(2\nu^2)$ and the value $0$ with probability $1-1/\nu^2$, where $\nu>1$.
Then, we obtain 
\begin{equation}\label{disc-dim1}
 T_{t,h}(\phi)(x)=\phi(x)+ 
\frac{b}{2\nu^2} 
\left(\phi(x+\sqrt{h}\nu )+\phi(x-\sqrt{h} \nu)-2 \phi(x)\right)\enspace ,
\end{equation}
with $b=1+\frac{1}{2} (A_{11}-1)(\nu^2-1)$.
This scheme is equivalent to an  explicit finite difference
discretization of~\eqref{HJB}
with a space step $\Delta x=\sqrt{h}\nu$. However it
is consistent with the Hamilton-Jacobi-Bellman
equation~\eqref{HJB} if and only if $b=A_{11}$ and so if
and only if $\nu=\sqrt{3}$.
In that case, the critical
condition $\frac{1}{2}(A_{11}-1)\leq 1$ is necessary for the
scheme to be monotone and it is equivalent to the
CFL condition $A_{11}h\leq (\Delta x)^2$.

For finite difference schemes, the CFL condition can be satisfied by
increasing $\Delta x$.
However, here $\Delta x$ is strongly connected to the possible values of $N$
and since the probability of large $N$ is small, one cannot 
avoid the critical constraint if we keep the
discretization~\eqref{defDcal2} of $\D_{t,h}^2(\phi)(x)$.

Let us consider now the same example in dimension $2$.
In that case, the
usual difficulty of finite difference schemes
is in the monotone discretization of mixed derivatives.
This can be solved for instance when 
the matrix $A$ is diagonally dominant by using only close points to the
initial point of the grid, that is  using
the 9-points stencil, see~\cite{kushner}, or in general
by using points of the grid which are far from the initial point 
(see for instance~\cite{bonnans,mirebeau}).
Here, we shall see that this difficulty is hidden by the critical
constraint $\frac{1}{2}\tr(A-I)\leq 1$, which implies in particular
that the matrix $A$ is diagonally dominant.

Indeed, consider the simple discretization of $N$ where
each entry of $N=(N_i)_{i=1,\ldots, 2}$ is replaced by a 
random variable as above, taking the values $\pm \sqrt{3}$ with probability
 $1/6$ and the value $0$ with probability $2/3$.
In that case, the critical constraint 
$\frac{1}{2}(A_{11}+A_{22}-2)\leq 1$ is necessary and sufficient for the
discretization to be monotone.
We have 
\begin{align*}
T_{t,h}(\phi)(x)=&\E\left(\phi(x+\sqrt{h} N)
(1+\frac{1}{2} \sum_{i,j=1}^2 (A_{ij}-\delta_{ij}) (N_iN_j -\delta_{ij}))\right)\\
=& \phi(x)\frac{2}{9} (2- \tr(A-I))\\
& +\frac{1}{18}\sum_{i=1}^2\sum_{\epsilon=\pm 1}\left(
\phi(x+\sqrt{3h} \epsilon e_i) ( 3 (A_{ii}-1)+2- \tr(A-I))
\right)\\
&+ \frac{1}{72} \sum_{\epsilon_1=\pm 1,\epsilon_2=\pm 1}\Big(
\phi(x+\sqrt{3h} (\epsilon_1 e_1+\epsilon_2 e_2))\\
&\qquad \Big(3 \Big(\sum_{i,j=1}^2(A_{ij}-\delta_{ij})\epsilon_i\epsilon_j\Big)
+2- \tr(A-I)\Big) \Big) \enspace .
\end{align*}
where $(e_1,e_2)$ is the canonical basis of $\R^2$.
This discretization can be rewritten as
\begin{align*}
T_{t,h}(\phi)(x)=&\phi(x)+ 
\frac{h}{2}
\left(\left(\sum_{i,j=1}^2 (A_{ij}-\delta_{ij} b) D_{ij}^h \phi(x)\right)
+ b \Delta^h \phi(x)\right) \enspace ,
\end{align*}
where $b= (1+ \tr(A-I))/3$, 
$ D_{ij}^h \phi$ is the standard $5$-point stencil discretization
of the partial derivative $\frac{\partial^2\phi}{\partial x_i\partial x_j}$
on the grid with space step $\Delta x=\sqrt{3h}$ (as above),
and $\Delta^h \phi$ is the discretization of $\Delta \phi$ using 
the external vertices of the $9$-point stencil (that is the points 
$x+\Delta x (\pm e_1+\pm e_2)$).
Note also that the critical constraint $\tr(A-I)\leq 2$ 
implies $b\leq 1$. Moreover, since $A-I$ is positive semidefinite,
then $2|A_{12}|\leq \tr(A-I)\leq 2$ and $A_{ii}\geq 1$, 
so $|A_{12}|\leq A_{ii}$ for $i=1,2$. The latter condition 
means that the matrix is diagonally dominant.
So, in dimension $2$, the critical constraint implies 
automatically that the equation can be discretized using a 
9-points stencil finite difference monotone scheme.

Hence, the difficulty of the probabilistic scheme does not come only or
necessarily from mixed derivatives as for finite difference schemes.
It essentially comes from the weights of the possible values of $N$,
which link strongly the possible space discretization 
and time discretization steps.
The new approximation of  $\D_{t,h}^2(\phi)$ that we propose in next section
will allows one to change these weights,
while keeping the probabilistic interpretation.

\section{A monotone probabilistic scheme for fully nonlinear PDEs}
\label{sec-monotone}

We denote by $\C^{k}$ the set of functions from $[0,T]\times\R^d$ to $\R$ 
with continuous partial derivatives up to order $k$ in $t$ and $x$,
and by $\C^{k}_{\text{b}}$ the subset of functions with bounded
such derivatives.

\begin{theorem}\label{theo2} 
Let $v\in\C^{4}_{\text{b}}$, $\hat{X}$ as in~\eqref{xhat},
$\Sigma\in \R^{d\times \ell}$ for some $\ell$ and denote
$A=\Sigma \Sigma\tp$.
Let $N$ be a one dimensional normal random variable.
For a nonnegative integer $k$, consider the polynomial 
$\Po_{\Sigma,k}$ of degree $4k+2$ in the variable $w\in\R^d$ defined by:
\begin{subequations}\label{defpoly2k}
\begin{gather}
 \Po_{\Sigma,k}(w)= c_k\sum_{j=1}^{\ell}([\Sigma\tp w]_{j})^{4k+2}\|\Sigma_{.j}\|_2^{-4k}-K \end{gather}
where, for any real vector $v\in\R^d$ and $j\leq d$, $[v]_j$ denotes 
the $j$th coordinate  of $v$, for 
any matrix $\Sigma\in \R^{d\times \ell}$ and $j\leq \ell$,
$\Sigma_{.j}$ denotes the $j$th column of $\Sigma$, and
\begin{gather}
c_k= \frac{1}{\E\left[N^{4k+4}-N^{4k+2}\right]}\enspace ,\qquad
K:=\frac{\tr(A)}{4k+2}=\frac{\sum_{j=1}^{\ell}\|\Sigma_{.j}\|_2^{2}}{4k +2}
\enspace .
\end{gather}
\end{subequations}
We have, for $(t,x)\in [0,T]\times \R^d$:
\begin{gather}
\E\left[\Po_{\Sigma,k}(h^{-1/2}(W_{t+h}-W_t))
 \mid \hat{X}(t)=x\right] = 0 \label{theo2-1}\\
h^{-1}\E\left[v(t+h,\hat{X}(t+h)) \Po_{\Sigma,k}(h^{-1/2}(W_{t+h}-W_t))
 \mid \hat{X}(t)=x\right]=\nonumber  \\
\frac{1}{2}\tr( \underline{\sigma}(x)\Sigma\Sigma\tp\underline{\sigma}\tp(x)\, D^2v(t,x))+O(h) \enspace , 
\label{theo2-2}
\end{gather}
where the error $O(h)$ is uniform in $t$ and $x$.
\end{theorem}
\begin{proof}[Sketch of proof]
Since $\Sigma\Sigma\tp=\sum_{j=1}^{\ell} \Sigma_{.j} \Sigma_{.j}\tp$, it is
sufficient to show~\eqref{theo2-2} when $A=\Sigma_{.j} \Sigma_{.j}\tp$
for  each $j=1,\ldots, \ell$.
In that case, with $j$ fixed, using a unitary matrix $U$ with 
$j$th column equal to $ \Sigma_{.j}/\|\Sigma_{.j}\|_2$, 
we obtain that~(\ref{theo2-1}-\ref{theo2-2}) are equivalent to the same 
equations for $A=U_{.j} U_{.j}\tp$.
Since $U$ is a unitary matrix,
$U\tp (h^{-1/2}(W_{t+h}-W_t))$ is a $d$-dimensional normal random vector,
and in particular its $j$th coordinate $[U\tp (h^{-1/2}(W_{t+h}-W_t))]_j$ is 
a normal random variable.
This implies in particular~\eqref{theo2-1}.
Applying a Taylor expansion of $v$ around $(t,x)$ up to order
$2$ and using the values of the moments of any $d$-dimensional normal 
random vector, we deduce~\eqref{theo2-2}.
\end{proof}
Since the above approximation depends on the matrix to which 
the second derivative is multiplied, we cannot apply it directly
as an argument of $\G$ as in~\eqref{def-th-m}, but need instead
to use the expression of $\G$ as a supremum of Hamiltonians
which are affine with respect to $r,p,\Gamma$ and apply~\eqref{theo2-2} to each
of these Hamiltonians.
In what follows, we shall present a scheme which combine at the same time
this idea and the one of~\cite{fodjo1}.
Note that the decomposition of the matrix involved in the expression
of the second order terms of the Hamiltonians 
as the product $A=\Sigma \Sigma\tp$ is used in a similar way to obtain
general monotone 
finite difference discretizations (see for instance~\cite{mirebeau}).

Let us decompose the Hamiltonian $\Ha^{m,u}$
of~\eqref{defHmu} as $\Ha^{m,u}=\Li^{m}+\G^{m,u}$ with 
\begin{align*}
 \Li^{m}(x, p,\Gamma):=&
\frac{1}{2} \tr\left(a^{m}(x)  \Gamma \right) +\underline{f}^{m}(x)\cdot p \enspace,
\end{align*}
and $a^{m}(x)=\underline{\sigma}^{m}(x) \underline{\sigma}^{m}(x)\tp$,
and denote by $\hat{X}^m$
the Euler discretization of the diffusion with generator $\Li^m$.
Note that, we can also choose a linear operator $\Li^{m}$ depending on $u$,
but this would increase too much the number of simulations.
We can also choose the same linear operator $\Li^{m}$ for different
values of $m$, which is the case in Algorithm~\ref{algo1} below.
Assume that $a^{m}(x)$ is positive definite (so that $\underline{\sigma}^{m}(x)$ is invertible) and that
$a^{m}(x)\leq \sigma^m(x,u) \sigma^m(x,u)\tp$, for all 
$x\in\R^d,\; u\in \U$, and denote by $\Sigma^m(x,u)$ any $d\times \ell$ matrix 
such that
\begin{equation}\label{assump1}
\sigma^m(x,u) \sigma^m(x,u)\tp -a^{m}(x)= \underline{\sigma}^{m}(x) \Sigma^m(x,u)\Sigma^m(x,u)\tp\underline{\sigma}^{m}(x)\tp
\enspace .
\end{equation}
Such a matrix $\Sigma^m(x,u)$ exists under the above assumptions since 
$ \underline{\sigma}^{m}(x)$ is invertible,
and $\underline{\sigma}^{m}(x)^{-1}(\sigma^m(x,u) \sigma^m(x,u)\tp -a^{m}(x)) (\underline{\sigma}^{m}(x))\tp)^{-1}$  is
a symmetric nonnegative matrix.
Indeed, one can choose $\Sigma^m(x,u)$ as the square root of
the latter matrix.
One may also use its Cholesky factorization 
in which zero columns are eliminated: this leads to 
a rectangular and triangular matrix $\Sigma^m(x,u)$
of size $d\times \ell$, where $\ell$ is the rank of the factorized matrix.
This is what we use in the practical implementation of Algorithm~\ref{algo}
below.

Define
\begin{equation}
\G^{m,u}_1(x, r,p):= (f^m(x,u)-\underline{f}^{m}(x)) \cdot p 
-\delta^m(x,u)r +\ell^m(x,u) \enspace,
\end{equation}
so that 
\[ 
\G^{m,u}(x, r,p,\Gamma)
= \G^{m,u}_1(x, r,p)+\frac{1}{2} \tr\left(\underline{\sigma}^{m}(x) \Sigma^m(x,u) \Sigma^m(x,u)\tp \underline{\sigma}^{m}(x)\tp\Gamma\right)\enspace .\]

Applying~\eqref{theo2-2} and~\eqref{defDcal}, we deduce the following result.

\begin{corollary}\label{cor-const}
Let $ \D_{t,h,m}^i(\phi)$, $i=0,1$,  be given by~{\rm (\ref{defDcal}-\ref{defpol})}, with $\underline{\sigma}^m$ 
and $\hat{X}^m$ instead of  $\underline{\sigma}$ and  $\hat{X}$ respectively.
Let $ \D_{t,h,m,\Sigma,k}^2(\phi)$ be defined as:
\[ \D_{t,h,m,\Sigma,k}^2(\phi)(x):=
h^{-1}\E\left[\phi(\hat{X}^m(t+h)) \Po_{\Sigma,k}(h^{-1/2}(W_{t+h}-W_t))
 \mid \hat{X}^m(t)=x\right]\]
with  $\Po_{\Sigma,k}$ as in~\eqref{defpoly2k}.

Consider the operator:
\begin{gather*}
T_{t,h}(\phi)(x)=   \max_{m\in\ M} \big\{ \D_{t,h,m}^{0}(\phi)(x)\\
+ h \max_{u\in \U}\big(
\G_1^{m,u}(x, \D_{t,h,m}^{0}(\phi)(x),\D_{t,h,m}^{1}(\phi)(x))
+  \D_{t,h,m,\Sigma^m(x,u),k}^2(\phi)(x)\big)\big\}
\enspace .
\end{gather*}
Assume that the maps $\sigma^m$, $f^m$ and $\Sigma^m$ are bounded with respect
to $u\in \U$.
Then, for $v\in\C^{4}_{\text{b}}$, $t\in\T_h$, and $x\in \R^d$, we have
\[ \frac{T_{t,h}(v(t+h,\cdot))(x)-v(t,x)}{h}=
\frac{\partial v}{\partial t} 
+\Ha(x, v(t,x), Dv(t,x), D^2v(t,x))+O(h)\enspace .\]
\end{corollary}
This result shows the consistency of the scheme~\eqref{scheme}
in the sense of~\cite{barles90}.
This implies easily that if the solution $v$ of~\eqref{HJB} is smooth enough, 
then the solution of~\eqref{scheme} converges to $v$ when $h$ goes to zero.
In the general case, when $v$ is only Lipschitz continuous for instance,
the convergence is obtained by the theorem of
Barles and Souganidis~\cite{barles90}. For this, 
one need to satisfy also the other assumptions of the theorem,
that we shall now show.

Note that when $k=0$, and $\Li^m=\Li$ does not
depend on $m$, the above operator $T_{t,h}$ coincides with the operator~\eqref{def-th-m} proposed in~\cite{touzi2011},
since $\D_{t,h,m,\Sigma,0}^2(\phi)(x)= 
\frac{1}{2}\tr( \underline{\sigma}(x)\Sigma\Sigma\tp\underline{\sigma}\tp(x)\, 
 \D_{t,h}^2(\phi)(x))$.
In~\cite[Lemma 3.12]{touzi2011}, the monotonicity of the scheme is
proved under the
critical constraint that $\tr(a(x)^{-1} \partial_{\Gamma} {\G})\leq 1$.
This constraint is equivalent to the condition that 
$\frac{1}{2}\tr(\Sigma^m(x,u) \Sigma^m(x,u)\tp)\leq 1$ for all $x\in\R^d$
and all useful controls $m\in\M$ and $u\in \U$ that are optimal 
in the expression of $\Ha$,
 which means that the constant $K$ in Theorem~\ref{theo2} is $\leq 1$
for $k=0$.
By increasing $k$, we can obtain that this constant $K$ is $\leq 1$ 
in more general situations, which implies the following
monotonicity result.

\begin{theorem}\label{theo-monotone}
Let $T_{t,h}$ be as in Corollary~\ref{cor-const}.
Assume that the map $\tr(\Sigma^m(x,u)\Sigma^m(x,u)\tp)$ is upper bounded 
in $x$ and $u$ and let $\bar{a}$ be  an upper bound.
Assume also that $\delta^m$ is upper bounded, 
and that there exists a bounded map $g^m$ (in $x$ and $u$) such that
$f^m(x,u)-\underline{f}^{m}(x)=\underline{\sigma}(x) \Sigma^m(x,u)g^m(x,u)$.
Then, for $k$ such that $\bar{a}< 4k+2$,
there exists $h_0$ such that
$T_{t,h}$ is monotone for $h\leq h_0$
over the set of bounded continuous functions $\R^d\to\R$,
and there exists $C>0$ such that
$T_{t,h}$ is $Ch$-almost monotone for all $h>0$.
\end{theorem}
\begin{proof}
Let $\phi,\psi:\R^d\to\R$ be bounded and let $h>0$.
We can write
$T_{t,h}$ as the supremum over $m\in \M$ and $u\in \U$ of the operators
\begin{align*} T_{t,h}^{m,u}(\phi)(x)= \; &  \D_{t,h,m}^{0}(\phi)(x)\\
&+ h \big(
\G_1^{m,u}(x, \D_{t,h,m}^{0}(\phi)(x),\D_{t,h,m}^{1}(\phi)(x))
+  \D_{t,h,m,\Sigma^m(x,u),k}^2(\phi)(x)\big)\\
=\; & \E\left[\phi(\hat{X}^m(t+h)) \Po^{h,m,u,x}(h^{-1/2}(W_{t+h}-W_t))
 \mid \hat{X}^m(t)=x\right]\\
&+h \ell^m(x,u)
\enspace ,
\end{align*}
where 
\begin{align*} \Po^{h,m,u,x}(w)=&1+h (f^m(x,u)-\underline{f}^{m}(x)) \cdot 
((\underline{\sigma}(x)\tp)^{-1} h^{-1/2} w)\\
&-h \delta^m(x,u) 
+\Po_{\Sigma^m(x,u),k}(w)
\enspace.
\end{align*}
If $L\geq 0$ is such that $-L$ is a lower bound of $\Po^{h,m,u,x}(w)$ for
all $m,u,x,w$, we obtain that the operators $T_{t,h}^{m,u}$ 
satisfy~\eqref{amonotone}, and taking the supremum, we get that
$T_{t,h}$ also satisfies~\eqref{amonotone} on the set
$\F$ of bounded continuous functions $\R^d\to\R$.
Let $C$ be an upper bound of all the $\delta^m$
and $\|g^m\|_2$ with $m\in \M$ 
(which is a finite set).
We get that 
\begin{align*} \Po^{h,m,u,x}(w)\geq &1-h^{1/2}C 
\|\Sigma^m(x,u)\tp w\|_2-h C
+\Po_{\Sigma^m(x,u),k}(w)
\enspace.
\end{align*}
For any matrix $\Sigma\in\R^{d\times \ell}$, %
$w\in\R^d$, and $\epsilon,\eta>0$, we have
\begin{align*}
 \|\Sigma\tp w\|_2 &\leq \frac{\epsilon}{2} \|\Sigma\tp w\|_2^2
+\frac{1}{2\epsilon}\\
&=   \frac{\epsilon}{2} \left(\sum_{j=1}^{\ell}\left(\frac{[\Sigma\tp w]_{j}}{\|\Sigma_{.j}\|_2}\right)^2 \|\Sigma_{.j}\|_2^{2 }\right)
+\frac{1}{2\epsilon}\\
&\leq   \frac{\epsilon\eta^{2k}}{4k+2} \left(\sum_{j=1}^{\ell}\left(\frac{[\Sigma\tp w]_{j}}{\|\Sigma_{.j}\|_2}\right)^{4k+2} \|\Sigma_{.j}\|_2^{2 }\right)
+\frac{2k\epsilon}{(4k+2)\eta}\left(\sum_{j=1}^{\ell} \|\Sigma_{.j}\|_2^{2 }\right)
+\frac{1}{2\epsilon}\\
&=   \frac{\epsilon\eta^{2k}}{4k+2} c_k^{-1}
\left( \Po_{\Sigma,k}(w)+\frac{\tr(\Sigma\Sigma\tp)}{4k+2}\right)
+\frac{2k\epsilon}{(4k+2)\eta}
\tr(\Sigma\Sigma\tp)
+\frac{1}{2\epsilon}\enspace ,
\end{align*}
with $c_k>0$ as in~\eqref{defpoly2k}.
Taking $\eta=\epsilon^2$ such that 
$h^{1/2}C\frac{\epsilon^{4k+1}}{4k+2} c_k^{-1}=1$
and using that $\tr(\Sigma^m(x,u)\Sigma^m(x,u)\tp)\leq \bar{a}$,
 we obtain 
\begin{align*} \Po^{h,m,u,x}(w)\geq &1-hC -
\frac{\bar{a}}{4k+2}-\frac{h^{1/2}C}{\epsilon} (\frac{2k}{4k+2}\bar{a}+\frac{1}{2}))\enspace .\end{align*}
Since $\frac{h^{1/2}C}{\epsilon} $ is a multiple of $h^{(2k+1)/(4k+1)}$, 
there exists a constant $C_k$ depending on $k$, such that
\begin{align*} \Po^{h,m,u,x}(w)\geq &L_{k,h}:=1-hC -
\frac{\bar{a}}{4k+2}-C_k h^{(2k+1)/(4k+1)}\enspace, \end{align*}
for all $w\in\R^d$.
Let us choose $k$ such that $\frac{\bar{a}}{4k+2}<1$.
We get that the lower bound $L_{k,h}$ of $\Po^{h,m,u,x}$ is nonnegative 
for $h\leq h_0$ for some $h_0>0$, which implies from the above remark that
$T_{t,h}$ satisfies~\eqref{amonotone} with $L=0$.
Then, for $h\geq h_0$, $C_k h^{(2k+1)/(4k+1)}/h\leq C'$ for some
constant $C'>0$, which implies that
$L_{k,h}\geq -h(C+C')$ for all $h>0$.
This shows that $T_{t,h}$ satisfies~\eqref{amonotone} with $L=(C+C')h$.
\end{proof}
Note that the boundedness of $g^m$ holds if $f^m-\underline{f}^m$
is bounded and $\sigma^m(\sigma^m)\tp -a$ is uniformly
lower bounded by a positive matrix.
Also, the continuity of the maps to which $T_{t,h}$ is applied is
not necessary, Borel measurability is clearly sufficient.
The Lebesgue measurability is also sufficient since $h>0$
and $a^m(x)$ is positive definite, so that if $N$ is negligible,
then $X^m(t+h)\not\in N$ a.e.
In the latter case the inequalities and suprema in~\eqref{amonotone}
are for the a.s. partial order.

\begin{remark}\label{remark-k1}
As explained in Section~\ref{remark-k}, the
critical constraint that $\tr(a(x)^{-1} \partial_{\Gamma} {\G})\leq 1$
is necessary even in dimension $1$, and comes from the weak
weights of large values of the increments of the Brownian motion
in the expression of the derivatives as conditional expectations
in~\eqref{defDcal2}.
Let us see what happens when increasing $k$ 
by considering the simple example of  Section~\ref{remark-k} in dimension $1$.
So consider the same linear Hamiltonian and same operator $\Li$.
Then, the operator of Corollary~\eqref{cor-const} satisfies in any dimension:
\begin{align}
T_{t,h}(\phi)(x)&=  \E\left(\phi(x+\sqrt{h} N)
(1+ \Po_{\Sigma,k}(N)) \right)\enspace .
\end{align}
with  $\Po_{\Sigma,k}$ as in~\eqref{defpoly2k}, $\Sigma$ such that
$A-I= \Sigma\Sigma\tp$ and $N$ a $d$-dimensional normal random variable.
If $d=1$ and $\ell=1$, we can rewrite  $\Po_{\Sigma,k}$ as:
\[ \Po_{\Sigma,k}(w)= \frac{\Sigma^{2}}{4k +2}
\left(\frac{w^{4k+2}}{\E\left[N^{4k+2}\right]} -1\right)\enspace . \]
If we replace $N$ in the two above expressions (for consistency) 
by the random variable taking the values $\pm \nu$ with probability
 $1/(2\nu^2)$ and the value $0$ with probability $1-1/\nu^2$, where $\nu>1$,
we obtain the same expression as in~\eqref{disc-dim1} but with
$b=1+\frac{1}{4k+2} (A_{11}-1)(\nu^{2}-1)$.
As in Section~\ref{remark-k}, \eqref{disc-dim1}
is equivalent to an  explicit finite difference
discretization of~\eqref{HJB}
with a space step $\Delta x=\sqrt{h}\nu$, which
is consistent with the Hamilton-Jacobi-Bellman
equation~\eqref{HJB} if and only if $b=A_{11}$ and so if
and only if $\nu=\sqrt{4k+3}$.
The condition in Theorem~\ref{theo-monotone}
 is  equivalent here to $A_{11}<4k+3$,
which is equivalent to the strict CFL condition $A_{11}h< (\Delta x)^2$.
The difference with the scheme of~\cite{touzi2011} is that
we can increase $\nu$, thus the ratio between  $\Delta x$ and $\sqrt{h}$,
by increasing $k$.
\end{remark}

In the sequel, we shall also need the following property which is
standard.
We shall say that an operator $T$ between any sets $\F$ and $\F'$
of partially ordered sets of
real valued functions, which are stable by the addition of a constant function
(identified to a real number),
is \NEW{additively $\alpha$-subhomogeneous} if 
\begin{equation}\label{def-hom}
\lambda\in\R,\lambda \geq 0,\; \phi\in \F\; 
 \implies T(\phi+\lambda)\leq T(\phi)+\alpha \lambda 
\enspace .\end{equation}

\begin{lemma}\label{lem-hom}
Let $T_{t,h}$ be as in Corollary~\ref{cor-const}.
Assume that $\delta^m$ is lower bounded in $x$ and $u$.
Then,  $T_{t,h}$ is additively $\alpha_h$-subhomogeneous
over the set of bounded continuous functions $\R^d\to\R$,
for some constant $\alpha_h=1+Ch$ with $C\geq 0$.
\end{lemma}
\begin{proof}
Take for $C$ a nonnegative upper bound of $-\delta^m$.
\end{proof}
With the monotonicity, the $\alpha_h$-subhomogeneity implies the 
$\alpha_h$-Lipschitz continuity of the operator, which allows one to show 
easily the stability as follows.
\begin{corollary}\label{cor-stability}
Let the assumptions and conclusion of Theorem~\ref{theo-monotone} and
Lemma~\ref{lem-hom} hold
and assume also that $\psi$ and $\ell^m$ are bounded.
Let us consider the function $v^h$ defined on $\T_h\times \R^d$ 
by~\eqref{scheme} with $T_{t,h}$ as in Corollary~\ref{cor-const}
and $v^h(T,x)=\psi(x)$ for all $x\in \R^d$.
Then, $v^h$ is bounded, which means that the scheme~\eqref{scheme} is stable.
\end{corollary}
\begin{proof}
The boundedness of $\ell^m$ implies that $\Ha(x,0,0,0)$ is bounded.
Applying Corollary~\ref{cor-const} to the zero function,
we deduce that the function
$|T_{t,h}(0)|$ is bounded  by the constant function $Ch$,
for some constant $C>0$.
From the conclusion of Theorem~\ref{theo-monotone}, 
one can choose $C>0$ such that \eqref{amonotone} holds with $L=Ch$ 
on the set $\F$ of bounded functions $\R^d\to\R$.
Let also $\alpha_h=1+Ch$ be as in Lemma~\ref{lem-hom}.
Applying \eqref{amonotone} and Lemma~\ref{lem-hom}
we obtain that if $K_{t+h}$ is a positive constant such
that  $|v^h(t+h,\cdot)|\leq  K_{t+h}$,  then
\begin{align*}
v^h(t,\cdot) & \leq Ch(2 K_{t+h})+ T_{t,h}(K_{t+h}) \\
&\leq 2 Ch K_{t+h}+T_{t,h}(0)+\alpha_h K_{t+h} \\
&\leq Ch +(1+3Ch) K_{t+h}\enspace .\end{align*}
By symmetry, we obtain that $K_t= Ch +(1+3Ch) K_{t+h}$ 
is an upper bound of $|v^h(t,\cdot)|$.
By induction, we get that $v^h$ is bounded by $(1+3Ch)^{T/h}/3$ which is bounded
independently of $h$. This implies the stability of the scheme.
\end{proof}

Applying the theorem
of Barles and Souganidis~\cite{barles90}, we obtain the convergence
of the scheme.

\begin{corollary}
Let the assumptions and notations of Corollary~\ref{cor-stability} hold.
Assume also that~\eqref{HJB} has a strong uniqueness property 
for viscosity solutions
and let $v$ be its unique viscosity solution.
Let us  extend $v^h$ on $[0,T]\times \R^d$ as a continuous 
and piecewise linear function with respect to $t$.
Then, when $h\to 0^+$, $v^h$ converges to $v$ locally 
uniformly in $t\in [0,T]$ and $x\in \R^d$.
\end{corollary}

Similar results can be proved, under different assumptions on the
Hamiltonian, for value functions that have a given growth such as
a quadratic growth (functions that are bounded above and below by a multiple
of the quadratic function $\|x\|^2+1$). 
We shall not discuss this here  although this is the type of results that 
are needed for the Hamilton-Jacobi-Bellman equations
involved in the next section, see Theorem~\ref{th-dist}.
Indeed, there are few such results in the literature for unbounded
value functions which make more difficult to show all the 
steps of the convergence proof. For instance one need to extend 
the theorem of Barles and Souganidis~\cite{barles90} in the 
context of functions with a given growth.
Let us mention that in~\cite{assellaou},
Assellaou, Bokanowski and Zidani show convergence and estimation
results for semilagrangian schemes for quadratic growth value functions.
Unfortunately, nor the results nor the steps of the proof can
be used in our context due to the special assumptions made there.

\section{The probabilistic max-plus method}\label{sec-maxplus}

The algorithm of~\cite{fodjo1} was based on the scheme~\eqref{scheme},
with $T_{t,h}$ as in Corollary~\ref{cor-const}, and $k=0$.
Here, we shall construct it similarly but with $k$ as in
Theorem~\ref{theo-monotone}.
The originality of the algorithm of~\cite{fodjo1} is that
instead of applying a regression estimation to compute
$\D_{t,h}^i(v^h(t+h,\cdot))$ by projecting the functions inside the conditional
expectation into a (large) finite dimensional
linear space of functions, we approximate  $v^h$ by a max-plus 
linear combination of basic functions (namely quadratic forms) and use 
the following distributivity property which generalizes
Theorem 3.1 of 
McEneaney, Kaise and Han~\cite{mceneaney2011}.

In the sequel, we denote $\W=\R^d$ and $\D$ 
the set of measurable functions from $\W$ to $\R$ 
with at most some given growth or growth rate (for instance with at most
exponential growth rate), assuming that it contains the constant functions.

\begin{theorem}[\protect{\cite[Theorem 4]{fodjo1}}]\label{main-theo}
Let $G$ be a monotone additively $\alpha$-subhomogeneous operator
from $\D$ to $\R$, for some constant $\alpha>0$.
Let $(Z,\A)$ be a measurable space, and let $\W$ be endowed with its
Borel $\sigma$-algebra. 
Let $\phi:\W \times Z\to\R$ be a measurable map such that for all $z \in Z$, $\phi(\cdot,z)$ is %
continuous and belongs to $\D$. 
Let $v\in\D$ be such that
$v(W)=\sup_{z\in Z} \phi(W,z)$. Assume that $v$ is continuous and 
bounded. %
Then,
\[ G(v)= \sup_{\bar{z}\in\overline{Z}}
G(\bar{\phi}^{\bar{z}})\] 
where $\bar{\phi}^{\bar{z}} :  \W\to \R,\; W \mapsto \phi(W,\bar{z}(W))$,
and 
\begin{align*}
\overline{Z}=&\{\bar{z}\, : \W \to Z,\; \text{measurable
and such that}\; \bar{\phi}^{\bar{z}} \in\D\}.
\end{align*}
\end{theorem}

To explain the algorithm, assume
that the final reward $\psi$ of the control problem
can be written as the supremum of a finite number 
of quadratic forms.
Denote $\Q_d=\Sy_d\times \R^d\times \R$ (recall that 
$\Sy_d$ is the set of symmetric $d\times d$
matrices) and let 
\begin{equation}\label{paramquad}
 q(x,z):= \frac{1}{2} x\tp Q x + b\cdot x +c, \quad
\text{with}\;\; z=(Q,b,c)\in \Q_d\enspace, \end{equation}
be the quadratic form with parameter $z$ applied to the vector $x\in\R^d$.
Then for $g_T=q$, we have
\[ v^h(T,x)=\psi(x)=\sup_{z\in Z_T} g_T(x,z)\]
where $Z_T$ is a finite subset of $\Q_d$.

The application of the operator  $T_{t,h}$ of Corollary~\ref{cor-const}
 to a (continuous)
function $\phi : \R^d \to \R, x\mapsto \phi(x) $ can be written, for each 
$x\in \R^d$, as
\begin{subequations}\label{def-T-G}
\begin{align}
T_{t,h}(\phi)(x)&=\max_{m\in \M}G_{t,h,x}^m(\tilde{\phi}^{m}_{t,h,x})\enspace , 
\end{align}
where 
\begin{align}
&S_{t,h}^m: \R^d\times \W \to \R^d, \; (x,W) \mapsto
S_{t,h}^m(x,W)= x + \underline{f}^m(x)h+ \underline{\sigma}^m(x) W
\enspace ,\label{def-Sm}\\
&\tilde{\phi}^{m}_{t,h,x}  =\phi(S_{t,h}^m(x,\cdot))\in \D
\quad \text{if}\; \phi\in \D\enspace ,
\label{def-tilde}
\end{align}
and $G_{t,h,x}^m$ is the operator from $\D$ to $\R$ given by
\begin{eqnarray}
\lefteqn{G_{t,h,x}^m(\tilde{\phi})
= D_{t,h,m,x}^{0}(\tilde{\phi})}\nonumber\\
&&+ h  \max_{u\in \U}\big(
\G_1^{m,u}(x, D_{t,h,m,x}^{0}(\tilde{\phi}),D_{t,h,m,x}^{1}(\tilde{\phi}))
+D_{t,h,\Sigma^m(x,u),k}^2(\tilde{\phi})  \big)\enspace ,
\label{defG}\end{eqnarray}
\end{subequations}
with 
\begin{align*}
&D_{t,h,m,x}^0(\tilde{\phi})=
\E(\tilde{\phi}(W_{t+h}-W_t))\enspace ,\\
&D_{t,h,m,x}^1(\tilde{\phi})=
\E(\tilde{\phi}(W_{t+h}-W_t) 
(\underline{\sigma}^m(x)\tp)^{-1}h^{-1} (W_{t+h}-W_t))\enspace ,\\
& D_{t,h,\Sigma,k}^2(\tilde{\phi})(x):=
h^{-1}\E\left[\tilde{\phi}(W_{t+h}-W_t) \Po_{\Sigma,k}(h^{-1/2}(W_{t+h}-W_t))\right]
\enspace ,
\end{align*}
$\G_1^{m,u}$ and $\Sigma^m(x,u)$ as in Section~\ref{sec-monotone}, 
and $\Po_{\Sigma,k}$ as in~\eqref{defpoly2k}.
Indeed, the Euler discretization  $\hat{X}^m$ 
of the diffusion with generator $\Li^m$ satisfies
\begin{equation}\label{def-hatxm}
\hat{X}^m(t+h)= S_{t,h}^m(\hat{X}^m(t), W_{t+h}-W_t)\enspace .
\end{equation}

Using the same arguments as for Theorem~\ref{theo-monotone} and 
Lemma~\ref{lem-hom},  one can obtain the stronger property
that for $h\leq h_0$, all the operators $G_{t,h,x}^m$ belong to the
class of monotone additively $\alpha_h$-subhomogeneous operators
from $\D$ to $\R$.
This allows us to apply Theorem~\ref{main-theo} and
thus the following result. %
\begin{theorem}[\protect{\cite[Theorem 2]{fodjo1},
compare with~\cite[Theorem 5.1]{mceneaney2011}}]\label{th-dist}
Consider the control problem of Section~\ref{sec-int}.
Assume that $\U=\R^p$ and that for each $m\in\M$,
$\delta^m$ and $\sigma^m$ are constant, $\sigma^m$ is nonsingular, $f^m$ 
is affine with respect to $(x,u)$, 
$\ell^m$ is quadratic with respect to $(x,u)$ and
strictly concave with respect to $u$, and that $\psi$
is the supremum of a finite number of quadratic forms.
Consider the scheme~\eqref{scheme}, 
with $T_{t,h}$ and $G_{t,h,x}^m$  as in~\eqref{def-T-G},
$\underline{\sigma}^m$ constant and nonsingular,
$\Sigma^m$ constant and nonsingular and $\underline{f}^m$ affine.
Assume that the %
operators $G_{t,h,x}^m$ belong to the
class of monotone additively ${\alpha}_h$-subhomogeneous operators
from $\D$ to $\R$, for some constant ${\alpha}_h=1+{C}h$ with 
${C}\geq 0$.
Assume also that the value function $v^h$ of~\eqref{scheme}
belongs to $\D$ and is locally 
Lipschitz continuous with respect to $x$.
Then, for all $t\in\T_h$, there exists a set $Z_t$ 
and a  map $g_t:\R^d\times Z_t\to \R$ such that
for all $z\in Z_t$, $g_t(\cdot, z)$ is a quadratic form and
\begin{equation}\label{supquadt}
 v^h(t,x)=\sup_{z\in Z_t} g_t(x,z)\enspace .\end{equation}
Moreover, the sets $Z_t$ satisfy
$ Z_t= \M\times 
\{\bar{z}_{t+h}:\W\to Z_{t+h}\mid \text{Borel measurable}\}$. %
\end{theorem}

Theorem~\ref{th-dist} uses the following property which was stated
in~\cite[Lemma~3]{fodjo1} without proof, 
and without the upper bound assumption.
Since counter examples exist when the upper bound assumption
is not satisfied, we are giving here the proof of the lemma.

\begin{lemma}[\protect{Compare with~\cite[Lemma~3]{fodjo1}}] \label{lemma-quad}
Let us consider the notations and assumptions of Theorem~\ref{th-dist}.
Let $\tilde{z}$  be a measurable function from $\W$ to $\Q_d$ 
and let $\tilde{q}_{x}$ denotes the measurable map
$\W\to\R,\; W \mapsto q(S_{t,h}^m(x,W), \tilde{z}(W) )$, 
with $q$ as in~\eqref{paramquad}.
Assume that there exists $\bar{z}\in\Q_d$ such that
$q(x, \tilde{z}(W))\leq q(x,\bar{z})$ for all $x\in \R^d$,
and almost all $W\in\W$,
and that $\tilde{q}_{x}$ belongs to $\D$, for all $x\in \R^d$.
Then, the function $x \mapsto G_{t,h,x}^m(\tilde{q}_{x})$
 is a quadratic function,
 that is it can be written as $q(x,Z)$ for some $Z\in \Q_d$.
\end{lemma}
\begin{proof} %
Since $S_{t,h}^m$ is linear with respect to $x$,
$\tilde{q}_{x}(W)$ is a quadratic function of $x$
the coefficients of which depend on $W$.
Then, due to the assumptions that
$\underline{\sigma}^m$ and $\Sigma^m$ are constant and nonsingular,
we get that $D_{t,h,m,x}^i(\tilde{q}_{x})$ with $i=0,1$, and
$D_{t,h,\Sigma^m(x,u),k}^2(\tilde{q}^m_{x})$ are quadratic functions of
$x$. 
Let $G_{t,h,x}^{m,u}(\tilde{\phi})$ denotes the expression 
in~\eqref{defG} without the maximization in $u$.
We get that  $G_{t,h,x}^{m,u}(\tilde{q}_{x})$ is of the form 
$K(x,u)+(Ax+Bu)\cdot D_{t,h,m,x}^1(\tilde{q}_{x})$,
where $K$ is a quadratic function of $(x,u)$,
strictly concave with respect to $u$ and 
$A$ and $B$ are matrices.
This also holds if we replace $\tilde{z}(W)$ by $\bar{z}$,
that is if we replace $\tilde{q}_{x}$ by $\tilde{Q}^m_{t,h,x}$,
where $Q$ is  the quadratic function $Q(x)=q(x,\bar{z})$.
However in that case, since $Q$ is deterministic,
$D_{t,h,m,x}^1(\tilde{Q}^m_{t,h,x})= 
\D_{t,h,m}^1(Q)(x)= \E( D Q(S_{t,h}^m(x,W_{t+h}-W_t)))$ 
which is an affine function of $x$, since $D Q$ is affine.
Therefore $G_{t,h,x}^{m,u}(\tilde{Q}^m_{t,h,x})$ is
a quadratic function of $(x,u)$, strictly concave with respect
to $u$, so its maximum over $u\in\U$ is a quadratic function of
$x$, that we shall denote by $P(x)$.

Since $G_{t,h,x}^m$ is assumed to be  monotone from $\D$ to $\R$,
we get that 
$G_{t,h,x}^m(\tilde{q}_{x})\leq G_{t,h,x}^m(\tilde{Q}^m_{t,h,x})=P(x)$.
Therefore for all $x\in \R^d$ and $u\in \U=\R^p$,
we obtain that 
$K(x,u)+(Ax+Bu)\cdot D_{t,h,m,x}^1(\tilde{q}_{x})
= G_{t,h,x}^{m,u}(\tilde{q}_{x})\leq P(x)$.
So $(Ax+Bu)\cdot D_{t,h,m,x}^1(\tilde{q}_{x})$ is a polynomial of degree 
at most $3$ in the variables $x_1,\ldots, x_d,u_1,\ldots , u_p$
upper bounded by a polynomial of degree at most $2$. 
Taking the limit when the $x_i$ and $u_j$ go to $\pm \infty$,
we deduce that all the monomials of degree $3$ have zero coefficients,
so that $(Ax+Bu)\cdot D_{t,h,m,x}^1(\tilde{q}_{x})$ is a quadratic function
of $(x,u)$.
Hence, $G_{t,h,x}^{m,u}(\tilde{q}_{x})$ is 
a quadratic function of $(x,u)$, strictly concave with respect
to $u$, which implies that its maximum  over $u\in\U$,
$G_{t,h,x}^{m}(\tilde{q}_{x})$, is a quadratic function of
$x$.
\end{proof}
\begin{proof}[Sketch of proof of Theorem~\ref{th-dist}]
Lemma~\ref{lemma-quad} shows in particular (and indeed uses) the property that
each operator $T_{t,h}^m$ such that
$T_{t,h}^m(\phi)(x)= G_{t,h,x}^m(\tilde{\phi}^{m}_{t,h,x})$
with $G_{t,h,x}^m$ as in~\eqref{defG},
sends a deterministic quadratic form into a quadratic form.
Since for any finite number of quadratic forms,
there exists a quadratic form which dominates them,
the assumptions of Theorem~\ref{th-dist} imply that
$\psi$ and then all the functions $v^h(t,\cdot)$ are
upper bounded by a quadratic form (recall that $\M$ is a finite set).
Then, applying Theorem~\ref{main-theo} to the maps $v^h(t,\cdot)$
and using Lemma~\ref{lemma-quad}, we get the representation formula~\eqref{supquadt}.
\end{proof}

In Theorem~\ref{th-dist}, as in~\cite[Theorem 5.1]{mceneaney2011},
the sets $Z_t$ are infinite for $t<T$.
If the Brownian process is discretized in space,
the set $\W$ can be replaced by a finite subset, 
and the sets $Z_t$ become finite.
Nevertheless, their cardinality increases at each time step as
$\card{Z_t}=\card{\M} \times( \card{Z_{t+h}})^p$ where $p$ is the cardinality of
the discretization of $\W$. 
Then, if all the quadratic functions generated in this way were 
different, we would obtain that
$\card{Z_0}=\card{\M}^{-1/(p-1)} \times(\card{\M}^{1/(p-1)} \card{Z_{T}})^{p^{T/h}}$ is doubly exponential with respect to the number of time
discretization points and more than exponential with respect to $p$.
Since the Brownian process is $d$-dimensional, one may need to discretize
it with a number $p$ of values which is exponential in the dimension $d$.
Hence, the computational time of the resulting method would  be 
worst than the one of a usual grid discretization.
In~\cite{mceneaney2011}, McEneaney, Kaise and Han 
proposed to apply a pruning method at each time step $t\in\T_h$
to reduce the cardinality of $Z_t$.
For this, they assume already that the function  $v^h$ is represented
as the supremum of the quadratic functions parameterized by 
a finite set $Z_t$ of $\Q_d$.
They show that pruning (that is eliminating elements
of $Z_t$) is optimal if one looks for a subset of $\Q_d$
with given size representing $v^h$ as the supremum
of the corresponding quadratic functions with a minimal measure of the
error. There, the measure of the error is the maximum of the integral of the 
difference of functions with respect to any probabilistic measure on
$\R^d$. Then, restricting the set of probabilistic measures to the
set of normal distributions, they propose to use LMI techniques 
to find the elements of $Z_t$ that can be eliminated.
However, whatever the number $N$ of quadratic functions used
at the end to represent $v^h$ at each time step is,
the computational time of the pruning method is at least
in the order of the cardinal of the initial set $Z_t$.
Hence, if $Z_t$ is computed as above using a discretization of the
Brownian process and the representation of $v^h$ at time
$t+h$ already uses $N$ quadratic forms, then $\card{Z_t}=\card{\M} \times  N^p$,
so that it is exponential with respect to $p$ and can then
be doubly exponential with respect to the dimension $d$.

In~\cite{fodjo1}, we proposed to compute the expression of 
the maps $v^h(t,\cdot)$ as a maximum of quadratic forms
by using simulations of the processes $\hat{X}^{m}$. %
These simulations are not only used for regression estimations
of conditional expectations, which are computed there only in the
case of random quadratic forms, leading to quadratic forms,
but they are also used to
fix the ``discretization points'' $x$ at which the optimal
quadratic forms in the expression~\eqref{supquadt}
are computed. 
We present below a particular case of the algorithm (in~\cite{fodjo1},
different samplings were tested for the regression). However,
we add the possibility of having the same operator $\Li^m$ 
for different $m$, 
in which case we choose to simulate the process $\hat{X}^m$ only
one time for each possible $\Li^m$,
then the number of simulations and quadratic forms decreases.
To formalize this, we consider in the algorithm the projection map $\pi$
which sends an element $m$ of $\M$ to a particular element of its 
equivalence class for the equivalence relation ``$m\sim m'$ if $\Li^m=\Li^{m'}$''.

\noindent
\hrulefill \
\begin{algorithm}[Compare with \protect{\cite[Algorithm1]{fodjo1}}]\ \label{algo1}
\label{algo}
\noindent
\emph{Input:} A constant $\eps$ giving the precision,
a time step $h$ and a horizon time $T$ such that $T/h$ is an integer,
a $3$-uple $N=(\Nzero,\NX,\NW)$ of integers giving the numbers of samples,
such that $\NX\leq \Nzero$, a subset $\Msmall\subset\M$ and a projection
map $\pi: \M\to\Msmall$.
A finite subset $Z_T$ of $\Q_d$ such that
$|\psi(x)-\max_{z\in Z_T} q(x,z)|\leq \eps$, for all $x\in\R^d$,
and $\card{Z_T}\leq \cardMs\times  \Nzero$.
The operators $T_{t,h}$, $S_{t,h}^m$ and $G_{t,x,h}^m$  as in \eqref{def-T-G}
for $t\in\T_h$  and $m\in \M$,
with $\Li^m$ (and thus  $S_{t,h}^m$) depending only
on $\pi(m)$.

\noindent
\emph{Output:}
The subsets $Z_t$ of $\Q_d$, for $t\in\T_h\cup\{T\}$, 
and the approximate value function
$v^{h,N}:(\T_h\cup\{T\})\times \R^d\to \R$.

\noindent
$\bullet$ \emph{Initialization:}
Let  $\hat{X}^{m}(0)=\hat{X}(0)$, for all $m\in \Msmall$,
where $\hat{X}(0)$ is random and independent of the Brownian process.
Consider a sample of $(\hat{X}(0),(W_{t+h}-W_t)_{t\in \T_h})$
of size $\Nzero$ indexed by $\omega\in \Omega_{\Nzero}:=\{1,\ldots, {\Nzero}\}$,
and denote, for each $t\in\T_h\cup\{T\}$,
$\omega \in \Omega_{\Nzero}$, and $m\in \Msmall$,
$\hat{X}^{m}(t,\omega)$ the value
of  $\hat{X}^{m}(t)$ induced by this sample satisfying~\eqref{def-hatxm}.
Define $v^{h,N}(T,x)=\max_{z\in Z_T} q(x,z)$,
for $x\in \R^d$, with $q$ as in~\eqref{paramquad}.

\noindent
$\bullet$ For $t=T-h,T-2h,\ldots, 0$ apply the following 3 steps:

(1) Choose a random sampling $\omega_{i,1},\; i=1,\ldots, \NX$
among the elements of $\Omega_{\Nzero}$
and independently a random sampling  $\omega'_{1,j}\; j=1,\ldots, \NW$
among the elements of $\Omega_{\Nzero}$,
then take the product of samplings,
that is consider $\omega_{(i,j)}=\omega_{i,1}$ and $\omega'_{(i,j)}=\omega'_{1,j}$
for all $i$ and $j$,
leading to $(\omega_{\ell},\omega'_{\ell})$ 
for $\ell\in \Omega_{\NXW}:=\{1,\ldots,\NX\}\times \{1,\ldots, \NW\}$.

Induce the sample $\hat{X}^{m}(t,\omega_{\ell})$ (resp.\ $(W_{t+h}-W_t)(\omega'_{\ell})$)
for $\ell\in \Omega_{\NXW}$
of $\hat{X}^{m}(t)$ with $m\in\Msmall$ (resp.\ $W_{t+h}-W_t$).
Denote by $\W^N_t\subset \W$
the set of $(W_{t+h}-W_t)(\omega'_{\ell})$ for $\ell \in \Omega_{\NXW}$.

(2) For each $\omega\in\Omega_{\Nzero}$ and $m\in \Msmall$,  
denote $x_t=\hat{X}^{m}(t,\omega)$ and
construct $z_t\in \Q_d$ depending on  $\omega$ and  $m$ as follows: 

(a) Choose 
$\bar{z}_{t+h}:\W^N_t\to Z_{t+h}\subset \Q_d$ such that,
for all $\ell\in \Omega_{\NXW}$, we have 
\begin{align*}
& v^{h,N}(t+h,S^m_{t,h}(x_t,(W_{t+h} -W_t)(\omega'_\ell)))\\
&\; = q\big(S^m_{t,h}(x_t,(W_{t+h} -W_t)(\omega'_\ell))
,\bar{z}_{t+h}((W_{t+h} -W_t)(\omega'_\ell))\big)\enspace .
\end{align*}
Extend $\bar{z}_{t+h}$ as a measurable map from $\W$ to $\Q_d$.
Let $\tilde{q}_{t,h,x}$ be the element of $\D$ given by
$W \in \W  \mapsto q(S^m_{t,h}(x, W), \bar{z}_{t+h}(W) )$. 

(b) For each $\bar{m}\in\M$ such that $\pi(\bar{m} )=m$,
compute an approximation of  
$x \mapsto G_{t,h,x}^{\bar{m}} (\tilde{q}_{t,h,x})$ by a linear regression 
estimation on the set of quadratic forms using the sample
$(\hat{X}^{m}(t,\omega_\ell), (W_{t+h} -W_t)(\omega'_\ell))$, with $\ell\in\Omega_{\NXW}$, and denote by $z_t^{\bar{m}}\in \Q_d$ the parameter of the resulting quadratic form.

(c) Choose $z_t\in \Q_d$ optimal among the  $z_t^{\bar{m}}\in \Q_d$ at the point
$x_t$, that is such that
$q(x_t, z_t)= \max_{\pi(\bar{m})=m}q(x_t, z_t^{\bar{m}})$.

(3) Denote by $Z_t$ the set of all the $z_t\in \Q_d$ obtained in this way,
and define
\[ v^{h,N}(t,x)=\max_{z\in Z_t} q(x,z)\quad
\forall x\in \R^d \enspace .\]

\end{algorithm}
\hrulefill \

Note that no computation is done at Step (3), which gives only a
formula (or procedure)
to be able to compute the value function at each time step $t$ and
point $x\in\R^d$ as a function of the sets $Z_t$.
This is what is done for instance to obtain plots.
In particular, the algorithm only stores the elements of $Z_t$ 
which are elements of $\Q_d$. 
Since $Z_t$ satisfy
$\card{Z_t}\leq \cardMs \times \Nzero$ for all $t\in\T_h$,
and $\Q_d$ has dimension $(d+1)(d+2)/2$,
the memory space to store the value function at a time step is
in the order of $\cardMs \times \Nzero \times d^2$,
so the maximum space complexity of the algorithm is
$O(\cardMs \times \Nzero \times d^2\times T/h)$.
Before computing the value function,
one need to store the values of all the processes,
with a memory space in $O(\cardMs \times \Nzero \times d\times T/h)$.
Moreover, the total number of computations at each time step
is in the order of 
$(\cardMs \times \Nzero)^2\times \NW  \times d^2
+ \cardM \times \Nzero \times(\NX\times \NW\times d^2+\NX\times d^5+d^6)$,
where the first term corresponds to step (a) and the second one to 
step (b).
Note also that $\NX$  can be chosen to be in the order of a polynomial 
in $d$ since the regression is done on the set of quadratic forms,
so in general the second term is negligible, and it is also worth 
to take $\cardMs$ small.

As recalled above, the map 
$x \mapsto G_{t,h,x}^{\bar{m}} (\tilde{q}_{t,h,x})$ is a quadratic form,
hence there is no loss in choosing to do a regression estimation 
over the set of quadratic forms.
Hence, as stated in \cite[Proposition~5]{fodjo1},
under suitable assumptions, we have the convergence
$\lim_{\Nzero, \NXW\to\infty} v^{h,N}(t,x)= v^h(t,x)$.

\section{Numerical tests}\label{sec-illust}

To illustrate our algorithm, we consider the problem of
evaluating the superhedging price of an option under uncertain 
correlation model
with several underlying stocks (the number
of which determines the dimension of the problem),
and changing sign cross gamma.
The case with two underlying stocks 
was studied first as an example in Section 3.2 
of~\cite{kharroubi2014}, where the method proposed is based on a regression on 
a process involving not only the state but also the (discrete) control.
In~\cite{fodjo1}, we tested our  algorithm with $\Msmall=\M$
on the same $2$-dimensional example. Here we shall consider the same example
with $\Msmall$ reduced to one element and then consider a similar one
with 5 stocks (so in dimension $5$).
Illustrations are obtained from a \verb!C++! implementation
of Algorithm~\ref{algo1}, which can easily be adapted to any model.

With the notations of the introduction, the problem 
has no continuum control, so $u$ is omitted, and 
for all  $m \in \M$, $f^{m}=0$ and $\delta^m=0=\ell^m$.
So it reduces to maximize
\begin{align*}
J(t, x, \mu) :=& \E \left[ \psi(\xi_T)  \mid \xi_t = x \right] \enspace .
\end{align*}
The dynamics is given by 
$d \xi_{i,s} = \sigma_i \xi_{i,s} d B_{i,s}$ where the $B_{i}$ are Brownians
with uncertain correlations: $\<dB_{i,s},dB_{j,s}>=[\mu_s]_{ij} ds$ 
with $\mu_s\in \cor$,  a subset of the set of 
positive symmetric matrices with $1$ on the diagonal.
This is equivalent to the condition that
\[ [\sigma^{m}(x)\sigma^{m}(x)\tp]_{ij}= \sigma_i x_i \sigma_j x_j m_{ij},
\quad \text{for}\; m\in \cor \enspace .\]
Here we assume that $\cor$ is the convex hull of 
a finite set $\M$.
Since the Hamiltonian of the problem is linear with respect to $m$, the maximum
over $\cor$ is the same as the maximum over $\M$, so we can assume 
that the correlations satisfy  $\mu_s\in \M$.
We consider the following final payoff:

\begin{minipage}{.5\textwidth}
\begin{align*}
& \psi(x)=\psi_1(\max_{i\;\text{odd}} x_i-\min_{j\;\text{even}} x_j),\; x\in\R^d,\\
& \psi_1(x) = (x-K_1)^+-(x-K_2)^+,\; x\in \R,\\
&x^+ = \max(x,0),\\
&K_1 < K_2.
\end{align*}
\end{minipage}
\hfill
\begin{minipage}{.45\textwidth}
\ \\[1em]
\begin{picture}(0,0)%
\includegraphics{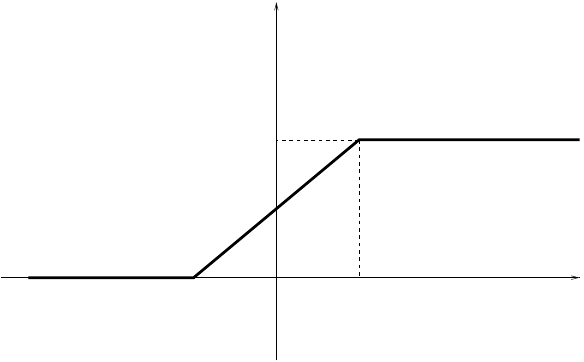}%
\end{picture}%
\setlength{\unitlength}{1160sp}%
\begingroup\makeatletter\ifx\SetFigFont\undefined%
\gdef\SetFigFont#1#2#3#4#5{%
  \reset@font\fontsize{#1}{#2pt}%
  \fontfamily{#3}\fontseries{#4}\fontshape{#5}%
  \selectfont}%
\fi\endgroup%
\begin{picture}(9506,5874)(1339,-7273)
\put(4201,-3661){\makebox(0,0)[lb]{\smash{{\SetFigFont{5}{6.0}{\rmdefault}{\mddefault}{\updefault}{\color[rgb]{0,0,0}$K_2-K_1$}%
}}}}
\put(4501,-6361){\makebox(0,0)[lb]{\smash{{\SetFigFont{5}{6.0}{\rmdefault}{\mddefault}{\updefault}{\color[rgb]{0,0,0}$K_1$}%
}}}}
\put(7201,-6361){\makebox(0,0)[lb]{\smash{{\SetFigFont{5}{6.0}{\rmdefault}{\mddefault}{\updefault}{\color[rgb]{0,0,0}$K_2$}%
}}}}
\end{picture}%
\end{minipage}

Since $\psi_1$ is nondecreasing, we have
$\psi(x)\geq \psi_1(x_i-x_j)$, for all $i$ odd and $j$  even.
Then, we can lower bound the value function in dimension $d$ by the
application of the value function of dimension $2$ and
volatilities  $(\sigma_i,\sigma_j)$ to $(x_i,x_j)$.

Note that all the coordinates of the controlled process stay in $\R_+$,
the set of positive real numbers.
To be in the conditions of Theorem~\ref{th-dist}, we approximate the function 
$\psi_1$  with a supremum of a finite number of quadratic forms on a large subset of $\R$, typically on $[-1000,1000]$,
so that $\psi$  is approximated with a supremum of a finite number of quadratic forms on the $x\in \R^d_+$ such that $x_i-x_j\in [-1000,1000]$.
Note that since the second derivative of $\psi_1$ is $-\infty$ in some points,
it is not  $c$-semiconvex for any $c>0$ and bounded domain, so
the approximation need to use some quadratic forms with a large negative 
curvature, and so we are not under the conditions of~\cite{mceneaney2011}. Moreover, since the state space is unbounded, one cannot 
approximate $\psi$ as a  supremum of a finite number of quadratic 
forms on all the state space as assumed in Algorithm~\ref{algo1}.
However, due to stability considerations,  the simulated
process stays with almost one probability in a ball around the
initial point, so that one may expect the value function 
to be well approximated 
in a bounded subset of $\R^d$.
The maps $\sigma^m$ for $m \in  \M$ are not constant but they are linear, 
and one can choose $\underline{\sigma}$ such that $\underline{\sigma}(x)^{-1}\sigma^m(x)$ is constant and $\underline{f}=0$, and get that
the result of Theorem~\ref{th-dist} still holds.

In the illustration below, we choose
$K_1=-5,\; K_2=5,\; T=0.25$, the time step $h=0.01$, 
the volatilities 
$\sigma_1=0.4,\; \sigma_2=0.3, \; \sigma_3=0.2, \sigma_4=0.3,\; \sigma_5=0.4$
and the following correlations sets:
\[ \M=\{m=\left[\begin{smallmatrix} 1 & m_{12}\\ m_{12} &1\end{smallmatrix}\right]\mid
m_{12}=\pm \rho\} 
\tag*{\text{for $2$ stocks,}} \]
and
\[ \M=\{m= \left[\begin{smallmatrix} 1 & m_{12}&0&0&0\\
 m_{12} &1&0&0&0\\ 0&0&1&0&0\\0&0&0&1&m_{45}\\
0&0&0&m_{45}&1 \end{smallmatrix}\right]\mid
m_{12}=\pm \rho,\; m_{45}=\pm\rho
\} \tag*{\text{for $5$ stocks.}} \]

In dimension $2$, we choose $\NX=10$, $\NW=1000$ and test several values 
of simulation size $\Nzero$, and compare our results with the true solutions 
that can be computed analytically when $\M$ is a singleton,
see Figures~\ref{fig1} and~\ref{fig2}.
For  $\rho=0$ or $\rho=0.4$, $k=0$ is sufficient in
Theorem~\ref{theo-monotone} (indeed $\G=0$ for $\rho=0$,
so there is no second derivative to discretize), whereas for  $\rho=0.8$,
one need to take $k=2$ to obtain the monotonicity of the scheme.
This may explain why a greater sampling size $\Nzero$ is needed
to obtain the convergence for $\rho=0.8$.

In dimension $5$, we choose $\NX=50$, $\NW=1000$ and $\Nzero=3000$,
and compare our results with a 
lower bound obtained from the results in dimension $2$, as
explained above, see Figure~\ref{fig3}.
Although, the lower bound appears to be above the value function 
computed from the Hamilton-Jacobi-Bellman equation in dimension 5,
the difference between the value function and the lower bound
is small and of the same amount as the difference observed in
Figure~\ref{fig2} between the value functions computed in dimension 2
with the simulation sizes $\Nzero=2000$  and $\Nzero=3000$.
This indicates that the size of the simulations  $\Nzero=3000$ is not enough
to attain the convergence of the approximation, although the results give
already the correct shape of the value function.
Such a result would be difficult to obtain with finite difference schemes,
and at least will take much more memory space.
For instance, the computing time  for one time step
of a finite difference scheme
on a regular grid over $[0,100]^5$ with $100$ steps by coordinate
is in $10^{10}$ and is thus
 comparable with the computing time of Algorithm~\ref{algo1},
$\Nzero^2\times \NW\times d^2$, with the above parameters,
whereas the memory space needed for the finite difference scheme
at each time step
is similar to the computing time and is thus  much larger than
the one needed in Algorithm~\ref{algo1} (in $\Nzero\times d^2=7.5\, 10^5$).

The computation of the value function in dimension $5$ took $\simeq 19$h
with the \verb!C++! program compiled with ``OpenMP''
on a $12$ core Intel(R) Xeon(R) CPU $E5-2667$ - $2.90$GHz
with $192$Go of RAM (each time iteration taking $\simeq 2500$s).
The main part of the computation time is taken by the
optimization part (a) of Algorithm~\ref{algo1}, with a
time in $O(\Nzero^2\times \NW\times d^2)$.
The bottleneck here is in the computation, for each given state $x$ at time 
$t+h$, of the quadratic form which is maximal in the
expression of $v^{h,N}(t+h,x)$.
Therefore, a better understanding of this maximization problem is necessary 
in order to decrease the total computing time.
 This would allow us to 
obtain better approximations in dimension 5 in particular,
and increase the dimension with a small cost.
Such an improvement is left for further work.

\figperso{\includegraphics[scale=0.30]{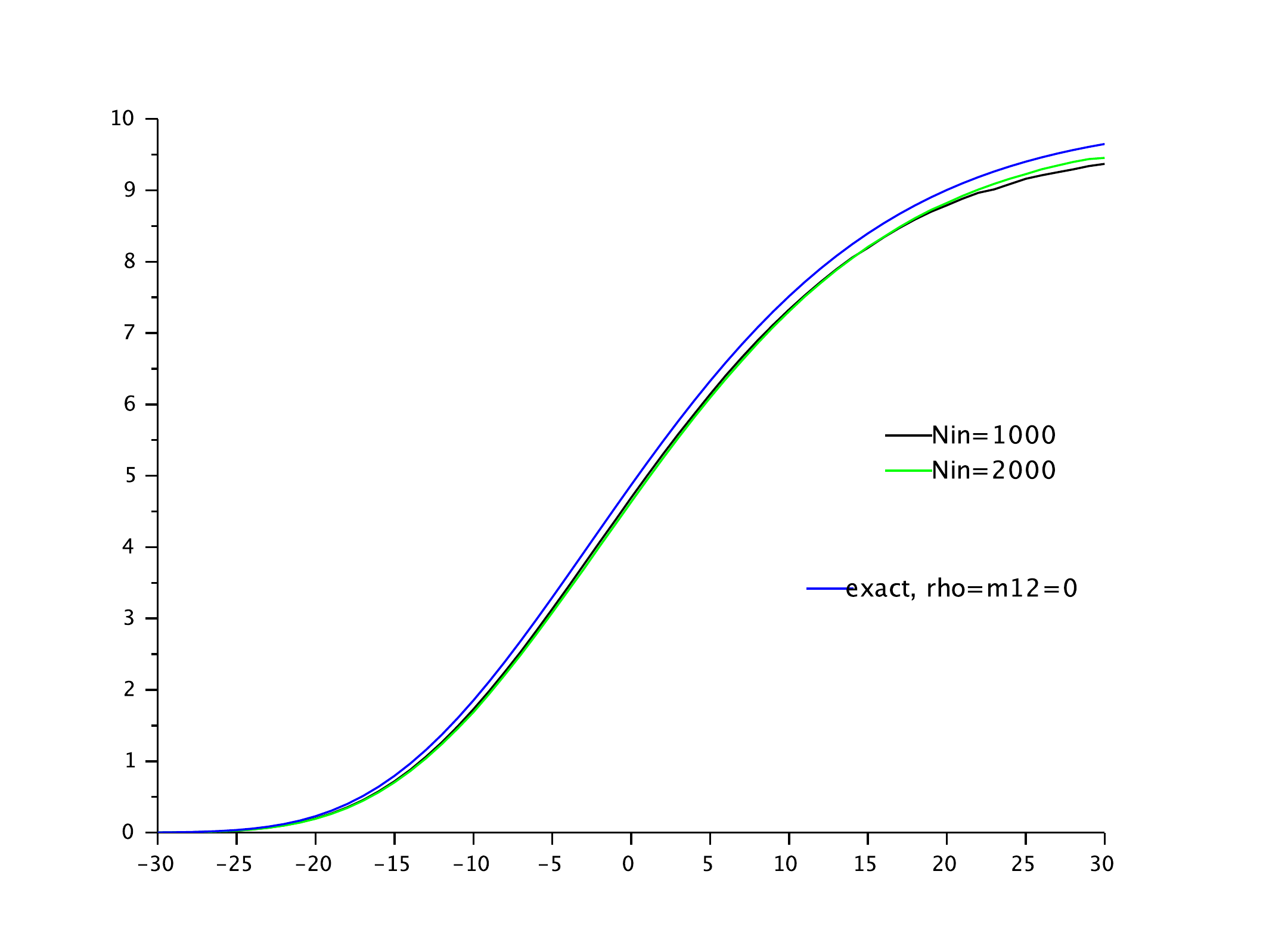}
\hfill \includegraphics[scale=0.30] {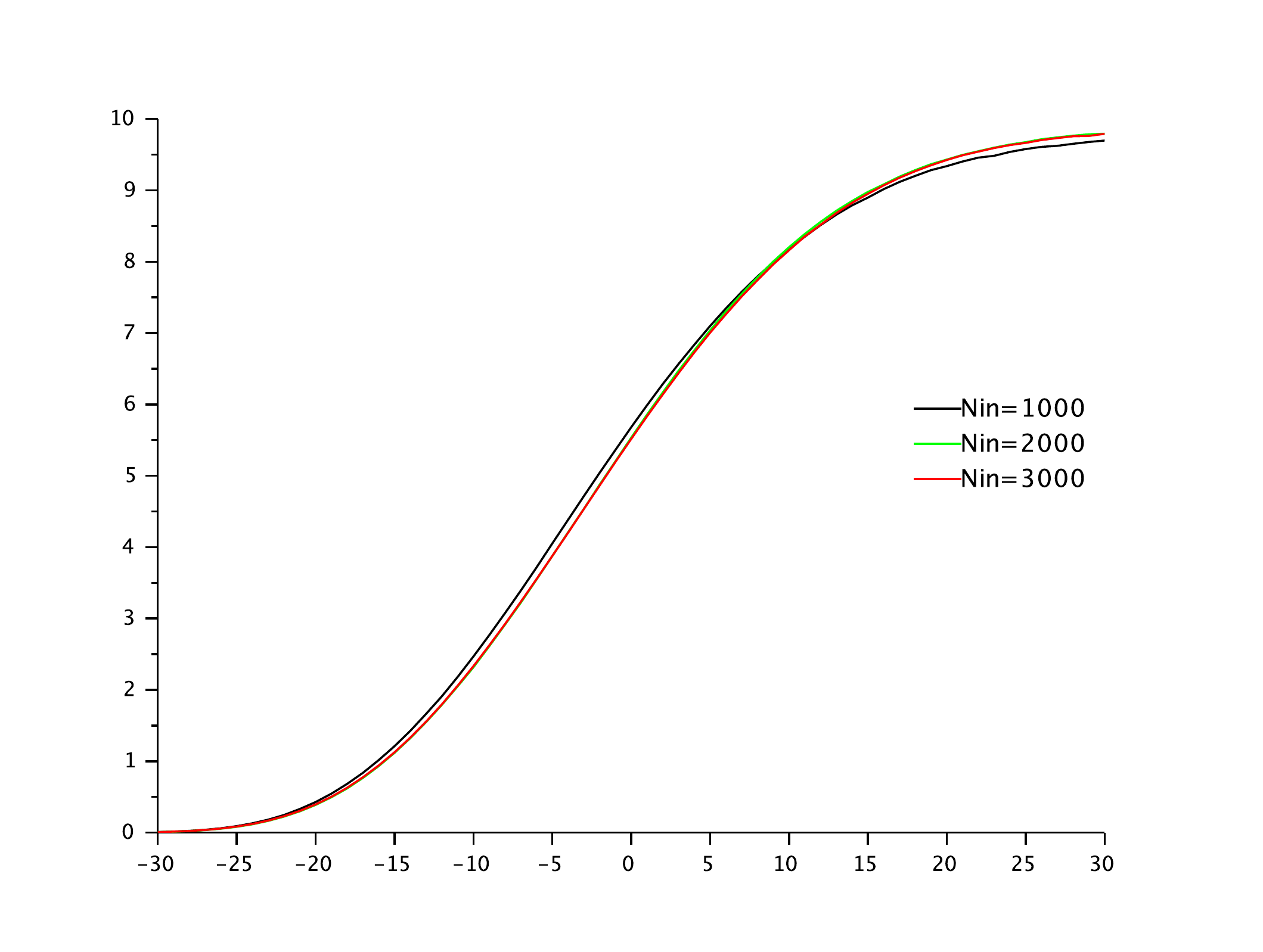} %
}{\caption{
Value function in dimension $2$, for $\rho=0$ on left,
and $\rho=0.4$ on right, at $t=0$, and $x_2=50$ as a function 
of $x_1-x_2$.
Here $\Nzero=1000,2000$, or $3000$, $\NX=10$, $\NW=1000$.
}\label{fig1}}

\figperso{\includegraphics[scale=0.30]{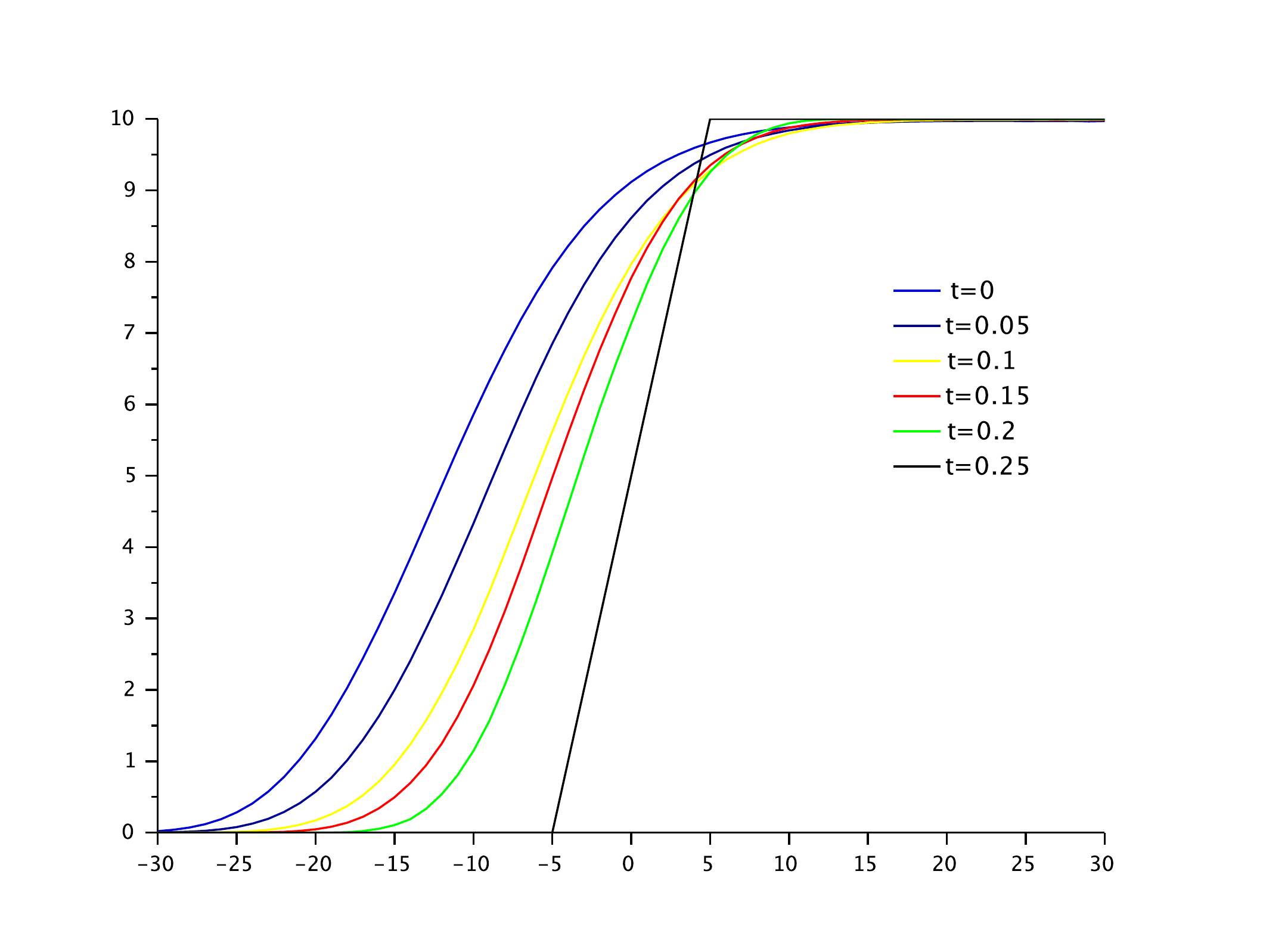} \hfill
\includegraphics[scale=0.30]{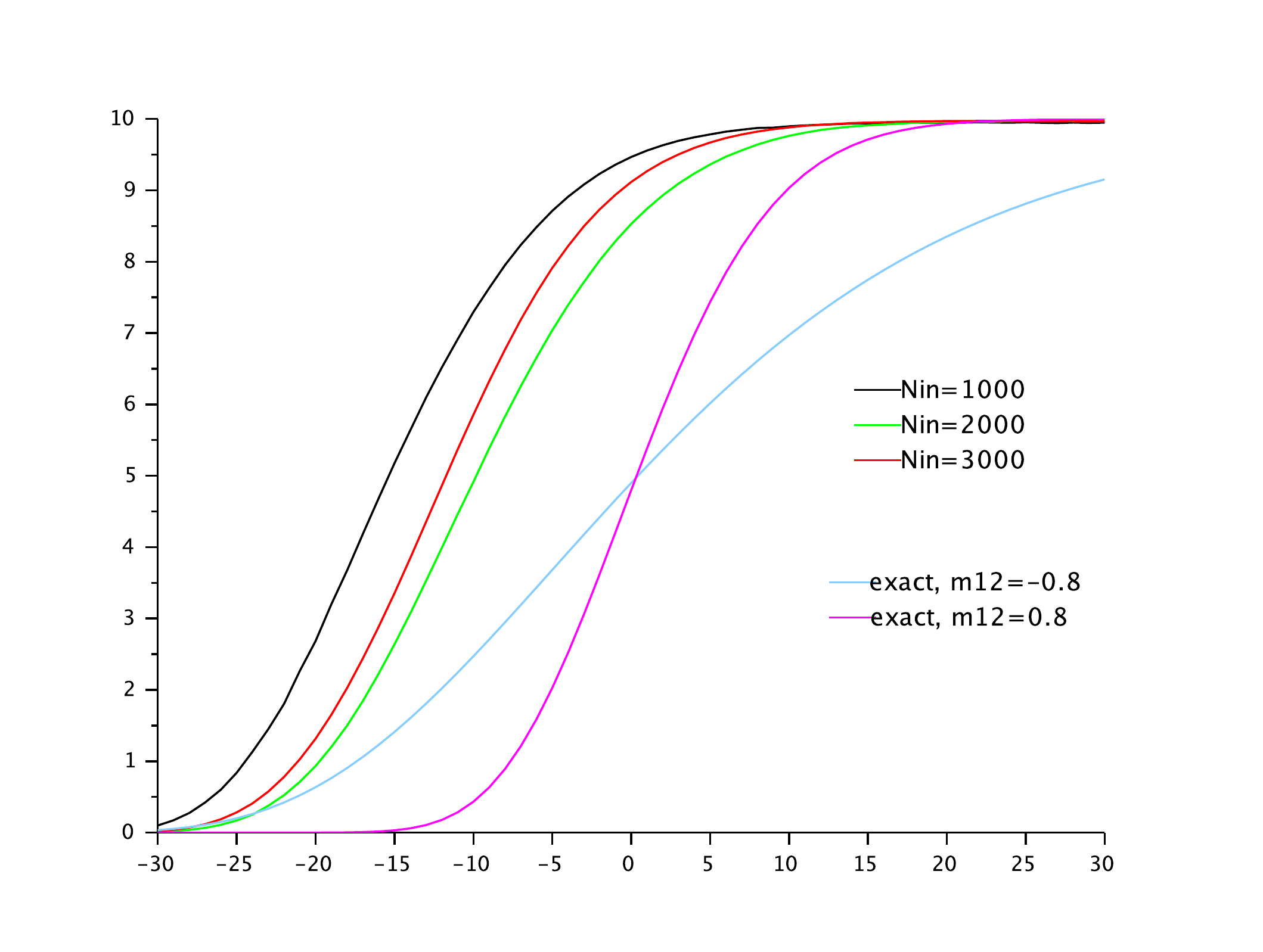}
}{\caption{
Value function in dimension $2$, for $\rho=0.8$, at $x_2=50$ as a function 
of $x_1-x_2$ obtained with $\NX=10$, $\NW=1000$.
On left, the value is shown at each time step multiple of $0.05$
and is obtained for $\Nzero=3000$.
On right, the value at time $t=0$ is compared for
$\Nzero=1000$, $2000$ and $3000$ and with the exact solution
when $\M$ is a singleton.
}\label{fig2}}

\figperso{\includegraphics[scale=0.3]{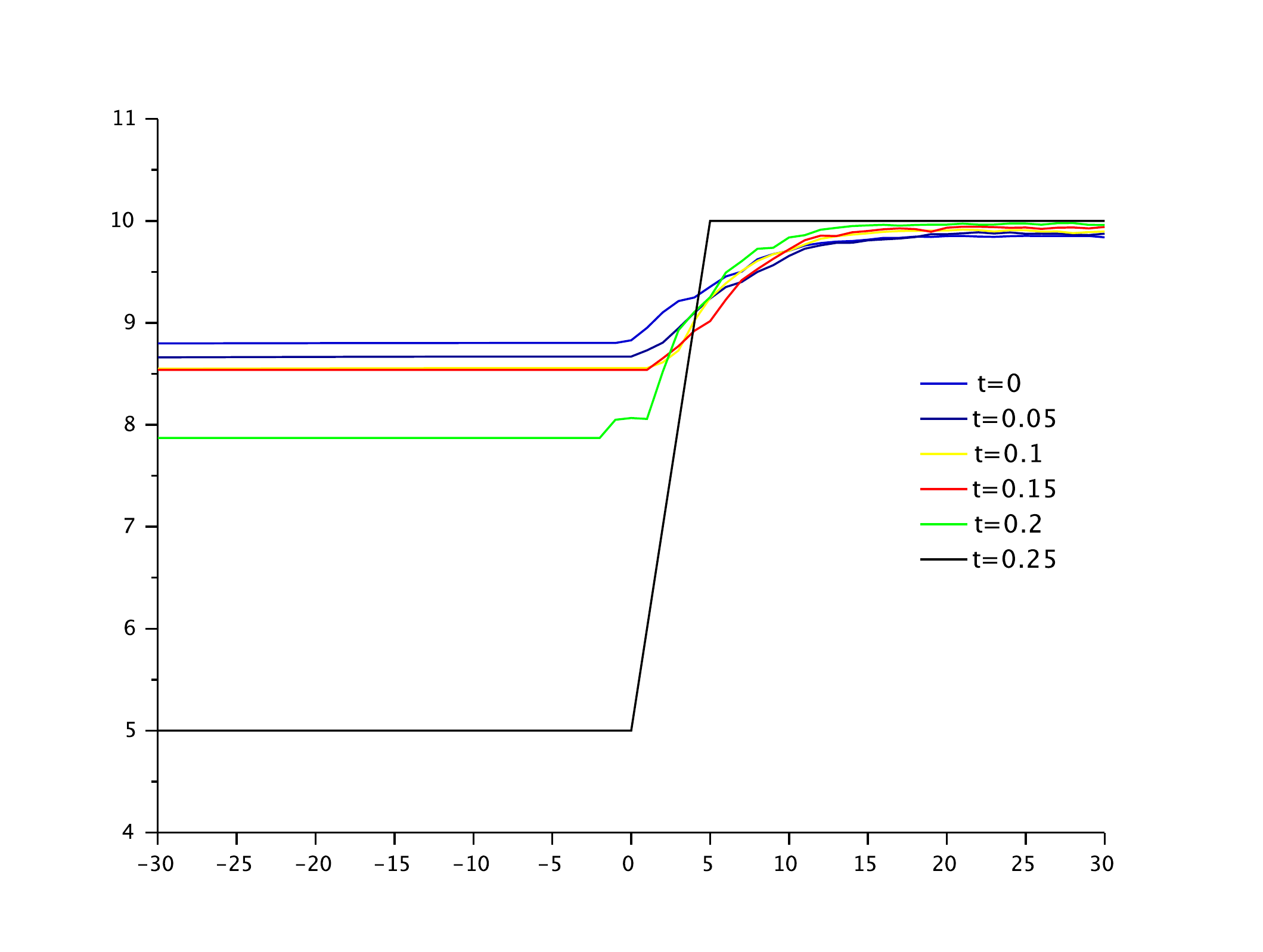}
\hfill \includegraphics[scale=0.3]{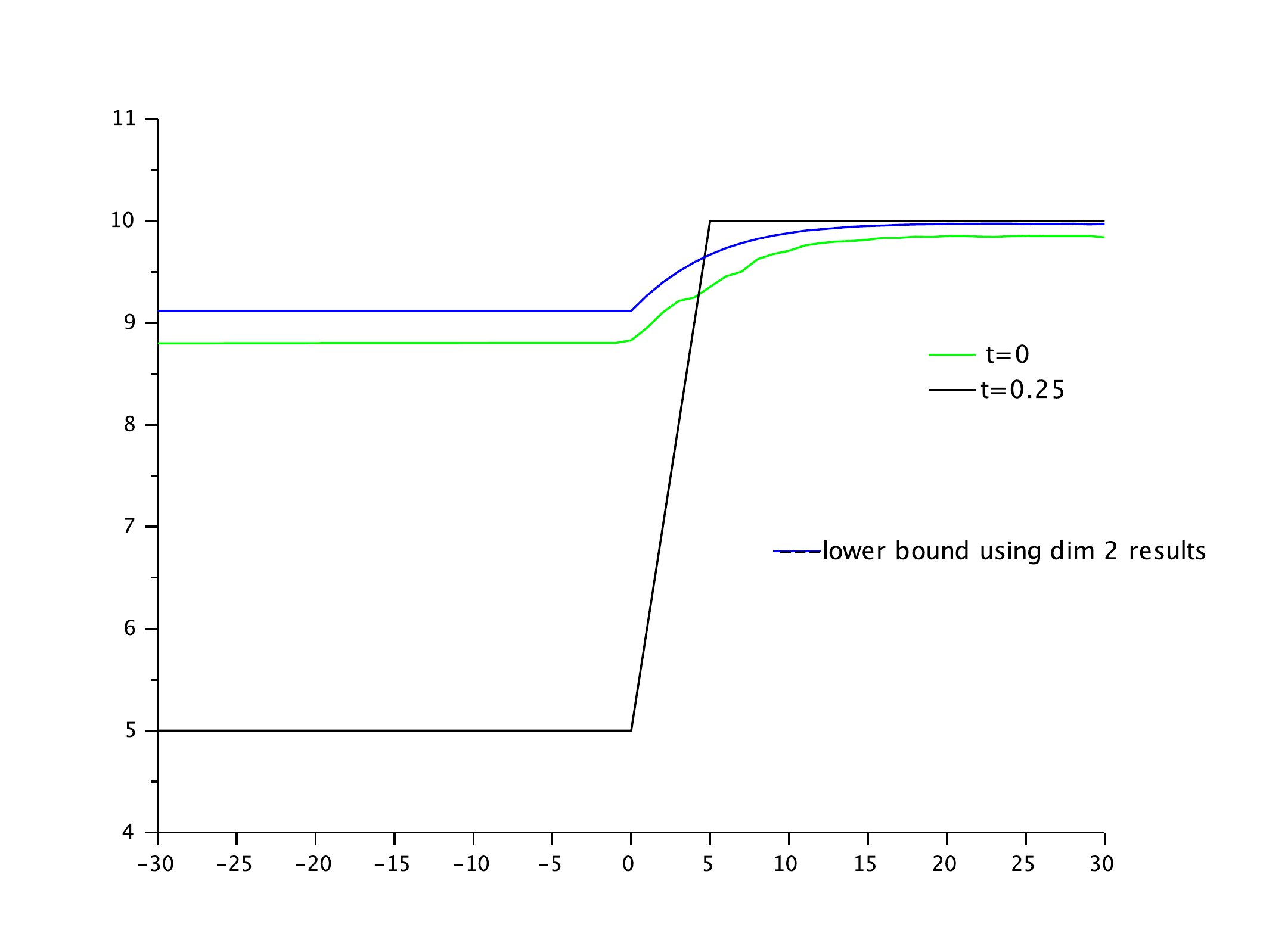}
}{\caption{
Value function for $\rho=0.8$ in dimension $5$, at 
 $x_2=x_3=x_4=x_5=50$ as a function of $x_1-x_2$.
Here $\Nzero=3000$, $\NX=50$, $\NW=1000$.
On left, the value is shown at each time step multiple of $0.05$.
On right, the value at time $t=0$ is compared with a lower bound obtained
by using the results in dimension $2$.
}\label{fig3}}

\bibliography{maxproba}
\bibliographystyle{plain}

\end{document}